\documentclass[11pt]{article}

\usepackage{amsmath,amsthm,amsfonts,amssymb,latexsym,amscd,enumerate,xypic}
\usepackage{extarrows, comment}

\theoremstyle{plain}
    \newtheorem{theorem}{Theorem}[section]
    \newtheorem{lemma}[theorem]{Lemma}
    \newtheorem{corollary}[theorem]{Corollary}
    \newtheorem{proposition}[theorem]{Proposition}
    \newtheorem{conjecture}[theorem]{Conjecture}

 \theoremstyle{definition}
    \newtheorem{definition}[theorem]{Definition}
    
    \newtheorem{example}[theorem]{Example}

    \newtheorem{remark}[theorem]{Remark}

\theoremstyle{remark}

\numberwithin{equation}{section}

\DeclareMathOperator{\Ad}{Ad}
\DeclareMathOperator{\ad}{ad}

 \DeclareMathOperator{\temp}{temp}

\DeclareMathOperator{\rank}{rank}

\DeclareMathOperator{\indx}{index}
\DeclareMathOperator{\im}{im}

\DeclareMathOperator{\Spin}{Spin}

\DeclareMathOperator{\SO}{SO}

\DeclareMathOperator{\SL}{SL}

\DeclareMathOperator{\SU}{SU}

 \DeclareMathOperator{\Ind}{Ind}

  \DeclareMathOperator{\Span}{span}

\begin{document}

\newcommand{\Spinc}{\Spin^c}
\newcommand{\Todo}{\textbf{To do}}

    \newcommand{\R}{\mathbb{R}}
    \newcommand{\C}{\mathbb{C}} 
    \newcommand{\N}{\mathbb{N}}
    \newcommand{\Z}{\mathbb{Z}} 
    \newcommand{\Q}{\mathbb{Q}}
    \newcommand{\bK}{\mathbb{K}} 
     \newcommand{\HH}{\mathbb{H}}

\newcommand{\g}{\mathfrak{g}}
\newcommand{\ka}{\mathfrak{a}}
\newcommand{\km}{\mathfrak{m}}
\newcommand{\kn}{\mathfrak{n}}
\newcommand{\kg}{\mathfrak{g}} 
\newcommand{\kt}{\mathfrak{t}}
\newcommand{\kA}{\mathfrak{A}}
\newcommand{\XX}{\mathfrak{X}}
\newcommand{\kh}{\mathfrak{h}} 
\newcommand{\kp}{\mathfrak{p}}
\newcommand{\p}{\mathfrak{p}}
\newcommand{\kk}{\mathfrak{k}}
\newcommand{\ks}{\mathfrak{s}}
\newcommand{\ku}{\mathfrak{u}}
\newcommand{\su}{\mathfrak{su}}

\newcommand{\cE}{\mathcal{E}}
\newcommand{\cA}{\mathcal{A}}
\newcommand{\calL}{\mathcal{L}}
\newcommand{\calH}{\mathcal{H}}
\newcommand{\cO}{\mathcal{O}}
\newcommand{\cB}{\mathcal{B}}
\newcommand{\cK}{\mathcal{K}}
\newcommand{\cP}{\mathcal{P}}
\newcommand{\calD}{\mathcal{D}}
\newcommand{\cF}{\mathcal{F}}
\newcommand{\calR}{\mathcal{R}}
\newcommand{\cX}{\mathcal{X}}
\newcommand{\calM}{\mathcal{M}}
\newcommand{\calS}{\mathcal{S}}
\newcommand{\cU}{\mathcal{U}}

\newcommand{\Sj}{ \sum_{j = 1}^{\dim G}}
\newcommand{\Sk}{ \sum_{k = 1}^{\dim M}}
\newcommand{\ii}{\sqrt{-1}}

\newcommand{\Bigwedge}{\textstyle{\bigwedge}}

\newcommand{\ddt}{\left. \frac{d}{dt}\right|_{t=0}}

\newcommand{\PM}{P}
\newcommand{\DM}{D}
\newcommand{\LM}{L}
\newcommand{\vM}{v}

\newcommand{\Wedge}{\lambda}

\newcommand{\specialin}{\hspace{-1mm} \in \hspace{1mm} }

\newcommand{\bspl}{\[ \begin{split}}
\newcommand{\espl}{\end{split} \]}

\newcommand{\Utilde}{\widetilde{U}}
\newcommand{\Xtilde}{\widetilde{X}}
\newcommand{\Dtilde}{\widetilde{D}}
\newcommand{\Etilde}{\widetilde{S}}
\newcommand{\wt}{\widetilde}
\newcommand{\pd}{\overline{\partial}}

\newcommand{\Rhat}{\widehat{R}}

\newcommand{\mattwo}[4]{
\left( \begin{array}{cc}
#1 & #2 \\ #3 & #4
\end{array}
\right)
}

\newcommand{\beq}[1]{\begin{equation} \label{#1}}
\newcommand{\eeq}{\end{equation}}

\title{A geometric formula for multiplicities of $K$-types of tempered representations}

\author{Peter Hochs, Yanli Song and Shilin Yu}
\date{\today}

\maketitle

\begin{abstract}
Let $G$ be a connected, linear, real reductive Lie group with compact centre. Let $K<G$ be compact. Under a condition on $K$, which holds in particular if $K$ is maximal compact, 
we give a geometric expression for the multiplicities of the $K$-types of  any tempered representation (in fact, any standard representation) $\pi$ of $G$. This expression is in the spirit of Kirillov's orbit method and the quantisation commutes with reduction principle. It is based on the geometric realisation of $\pi|_K$ obtained in an earlier paper. This expression was obtained for the discrete series by Paradan, and for tempered representations with regular parameters by Duflo and Vergne. We obtain consequences for the support of the multiplicity function, and a criterion for multiplicity-free restrictions that applies to general admissible representations. As examples, we show that admissible representations of $\SU(p,1)$, $\SO_0(p,1)$ and $\SO_0(2,2)$ restrict multiplicity-freely to maximal compact subgroups.
\end{abstract}

\tableofcontents


\section{Introduction}

\subsection{Background and motivation}

Let $G$ be a connected, linear, real reductive Lie group with compact centre. Let $K<G$ be a maximal compact subgroup. A tempered representation  of $G$ is a unitary irreducible representation whose $K$-finite matrix coefficients are in $L^{2+ \varepsilon}(G)$ for all $\varepsilon > 0$. 
The set $\hat G_{\temp}$ of these representations features in the Plancherel decomposition
\[
L^2(G) = \int^{\oplus}_{\hat G_{\temp}} \pi \otimes \pi^*\, d\mu(\pi)
\]
as a representation of $G\times G$, where $\mu$ is the Plancherel measure. Tempered representations are also important because they are used in the Langlands classification \cite{Langlands89} of admissible irreducible representations.

The restriction $\pi|_K$ of  a tempered representation $\pi$ to $K$ is determined by the multiplicities of all irreducible representations of $K$ in $\pi|_K$, i.e.\ the multiplicities of the $K$-types of $\pi$.
This restriction
contains a good deal of information about $\pi$. 
For example, if $\pi$ has real infinitesimal character, then Vogan showed that it is determined by its lowest $K$-type (see Theorem 8.1 in \cite{Vogan00}). 

If $\pi$ belongs to the discrete series, then Blattner's formula (proved by Hecht--Schmid \cite{Hecht75} and later also in \cite{DHV}) is an explicit combinatorial expression for the multiplicities of the $K$-types of $\pi$. For general tempered representations, there exist algorithms to compute these multiplicities. See for example the ATLAS software package\footnote{See \texttt{http://www.liegroups.org/software/}.} and its documentation \cite{ATLASdoc}. This involves representations of disconnected subgroups of $G$, which cannot be classified via Lie algebra methods. That is one of the reasons why it is a challenge to deduce general properties of multiplicities of $K$-types of tempered representations from such algorithms. Another reason is the cancellation of terms, that already occurs in Blattner's formula. That can make it hard, for example, to determine which multiplicities are zero.

Paradan \cite{Paradan03} gave a geometric expression for the multiplicities of the $K$-types of discrete series representations $\pi$. This was based on a version of the \emph{quantisation commutes with reduction} principle for a certain class of noncompact $\Spinc$-manifolds, and a geometric realisation of $\pi|_K$ based in turn on Blattner's formula and index theory of transversally elliptic operators. The main result in this paper, Theorem \ref{thm mult form}, is a generalisation of Paradan's result to arbitrary tempered representations. (In fact, it applies more generally to standard representations.) This generalisation is now possible, because of a {quantisation commutes with reduction} result for general noncompact $\Spinc$-manifolds proved recently by the first two authors of this paper \cite{HS16}. Theorem \ref{thm mult form} can in fact be generalised to more general compact subgroups $K<G$; see Corollary \ref{cor mult form K'}.
For tempered representations with regular parameters, the multiplicity formula was proved by Duflo and Vergne \cite{Duflo11}, via very different methods.
Our result has applications to multiplicity-free restrictions of admissible representations.

\subsection{The main result}

In Theorem \ref{thm mult form}, we use a homogeneous space of the form $G/H$, for a Cartan subgroup $H<G$ (depending on $\pi$). This can be identified with a coadjoint orbit $\Ad^*(G)\nu \subset \kg^*$ through a regular element $\nu$ (depending on $\pi$) of the dual of the Lie algebra $\kh$ of $H$. (The Lie algebra of a Lie group is denoted by the corresponding lower case Gothic letter.) First, assume that $\pi$ is induced from a discrete series representation of a factor $M$ in a cuspidal parabolic subgroup $MAN <G$. Then $G/H \cong \Ad^*(G)\nu$. Consider the map
\[
\Phi\colon G/H \xrightarrow{\cong} \Ad^*(G)\nu \to \kk^*.
\]
This is a \emph{moment map} in the sense of symplectic geometry, although we will need to work with the more general $\Spinc$-geometry. 

Let $\delta$ be an irreducible representation of $K$, and let $\eta$ be its highest weight for a maximal torus $T<K$ and a fixed positive root system for $(\kk, \kt)$. Let $\rho^K$ be half the sum of these positive roots. The \emph{reduced space} $(G/H)_{(\eta + \rho^K)/i}$ is
\[
(G/H)_{(\eta + \rho^K)/i} := \Phi^{-1}((\eta + \rho^K)/i)/T.
\]
This is a compact space, and
if $(\eta + \rho^K)/i$ is a regular value of $\Phi$ then it is an orbifold. In that case, it has a $\Spinc$-structure, induced by a given $K$-equivariant $\Spinc$-structure on $G/H$ (depending on $\pi$). The index of the corresponding $\Spinc$-Dirac operator is denoted by
\[
\indx((G/H)_{(\eta + \rho^K)/i}) \quad \in \Z.
\]
This can be computed via Kawasaki's index theorem,  formula (7) in \cite{Kawasaki81}. If $(\eta + \rho^K)/i$ is a singular value of $\Phi$, then Paradan and Vergne \cite{Paradan14} showed how to still define this index in a meaningful way, essentially by replacing $(\eta + \rho^K)/i$ by a nearby regular value; see Subsection \ref{sec ind red}. Our main result, Theorem \ref{thm mult form} is the following.
\begin{theorem}\label{thm main intro}
We have
\[
[\pi|_K:\delta] = \pm \indx((G/H)_{(\eta + \rho^K)/i}). 
\]
\end{theorem}
See Section \ref{sec mult form} for precise definitions of the sign $\pm$, for the dependence on $\pi$ of $H$, $\nu$ and the $\Spinc$-structure on $G/H$, and for the definition of the index on the right hand side. In fact, Theorem \ref{thm main intro} applies more generally to standard representations $\pi$; see Remark \ref{rem standard reps}.

If $\pi$ is not induced from a discrete series representation of $M$, then its infinitesimal character is singular. In this case, the natural map $G/H \to \Ad^*(G)\nu$ is a fibre bundle. We then use a different map $\Phi$ to define reduced spaces (see Subsection \ref{sec tempered reps} for details.). This map depends on choices made, but the end result does not: Theorem \ref{thm main intro} still holds in this case.

Theorem \ref{thm main intro}, and the results that follow, are  in fact true for more general compact subgroups $K<G$: it is sufficient if the map $\Phi$ is proper. (This is true if $K$ is maximal compact; see (1.3) in \cite{Paradan99}.) See Corollary \ref{cor mult form K'}. Duflo and Vargas showed that in the case of a discrete series representation $\pi$,  properness of $\Phi$ with $K$ replaced by a possibly noncompact, closed, reductive subgroup $H<G$ is equivalent to the restriction of $\pi$ to $H$ being admissible (i.e.\ decomposing into irreducibles with finite multiplicities); see Proposition 4 in \cite{DV10}.

In the case where $\pi$ is induced from the discrete series, Duflo and Vergne \cite{Duflo11} proved a multiplicity formula for its $K$-types analogous to Theorem \ref{thm main intro}. The parametrisation part of the orbit method used by Duflo and Vergne to prove their result is the one described in Section III of \cite{Duflo82}. The geometric/representation theoretic input is Kirillov's character formula, proved by Rossmann \cite{Rossmann78}; see also \cite{Vergne79}. Our approach to proving Theorem \ref{thm main intro} is based on the geometric realisation of $\pi|_K$ in \cite{HSY1}, and allows us to prove it in general, 
 i.e.\ even for tempered representations induced from limits of the discrete series. Furthermore, our result has applications to multiplicity-free restrictions of general admissible representations.



Theorem \ref{thm main intro}
allows us to use the geometry of $G/H$, or of the coadjoint orbit $\Ad^*(G)\nu$, to draw conclusions about the general behaviour of the multiplicities of the $K$-types of $\pi$. One such conclusion is about the support of the multiplicity function of the $K$-types of $\pi$. 
\begin{corollary}\label{cor intro 0}
All $K$-types of $\pi$ have highest weights in the set
\[
i\Phi(G/H) \cap i\kt^* - \rho^K.
\]
\end{corollary}
In fact, these highest weights even lie in the relative interior of this set, see Corollary \ref{cor zero}.

Applications of Theorem \ref{thm main intro} to multiplicity-free restrictions are described in Subsection \ref{sec intro mult free}.

\subsection{The orbit method and quantisation commutes with reduction}

Theorem \ref{thm main intro} is directly related to Kirillov's orbit method and Guillemin and Sternberg's quantisation commutes with reduction principle \cite{Guillemin82}. Indeed, if a representation $\pi$ of $G$ is associated to a coadjoint orbit $\cO_{\pi} \subset \kg^*$, and an irreducible representation $\delta$ of a closed subgroup $H<G$ is associated to a coadjoint orbit $\cO_{\delta} \subset \kh^*$, then according to this principle, one  expects that
\beq{eq QR orbit}
[\pi|_H:\delta] = Q\bigl( (\cO_{\pi} \cap p^{-1}(\cO_{\delta}) )/H \bigr),
\eeq
where $p\colon \kg^* \to \kh^*$ is the restriction map and $Q$ denotes some notion of geometric quantisation. In fact, a result of this form by Heckman \cite{Heckman82} for compact Lie groups was inspiration for Guillemin and Sternberg to develop the idea that quantisation commutes with reduction. The equality \eqref{eq QR orbit} is also related to the role that the Corwin--Greenleaf multiplicity function plays in the study of multiplicity-free restrictions (see below). 

In the setting of Theorem \ref{thm main intro}, suppose that the infinitesimal character $\chi$ of $\pi$ is a regular element of $i\kh^*$. Then it was shown in \cite{HSY1} that 
\[
\pi|_K = Q_K(\cO_{\pi}),
\]
where $\cO_{\pi} = \Ad^*(G)(\chi + \rho^{G,M})$, for an element $\rho^{G,M} \in i\kh^*$ defined in terms of half sums of positive roots (see \eqref{eq def rhoGM}), and where $Q_K$ stands for a natural notion of $K$-equivariant geometric quantisation of noncompact $\Spinc$-manifolds \cite{HS16, Zhang14, Paradan03, Paradan11, Vergne06}. If $H=K$, and $\delta \in \hat K$ has highest weight $\eta$ (hence infinitesimal character $\eta+ \rho^K$), then $\cO_{\delta} = \Ad^*(K)(\eta+\rho^K)$ for a $\Spinc$-version of geometric quantisation \cite{Paradan15}. Then Theorem \ref{thm main intro} is precisely the equality \eqref{eq QR orbit}, where $Q$ is given by the index of $\Spinc$-Dirac operators.

We have mentioned $\Spinc$-quantisation several times so far. Paradan showed in \cite{Paradan03} that it is natural to use a $\Spinc$-version of geometric quantisation to obtain multiplicities of $K$-types of representations of $G$, rather than the symplectic version. Paradan and Vergne showed in \cite{Paradan14} that the quantisation commutes with reduction principle has a natural extension to the $\Spinc$-setting. This was generalised to a result for noncompact $\Spinc$-manifolds in \cite{HS16} (see Theorem \ref{thm QR=0}), which we will use to prove Theorem \ref{thm main intro}.

If the infinitesimal character $\chi$ is singular, then the link between Theorem \ref{thm main intro} and the orbit method is less direct. Rather than using nilpotent coadjoint orbits in that case, we use $G/H$ as a desingularisation, which allows us to still obtain an expression for multiplicities of $K$-types.

\subsection{Multiplicity-free restrictions} \label{sec intro mult free}

The problem of determining when the restriction of an irreducible representation $\pi$ of $G$ to a closed subgroup $H$ is multiplicity-free is the subject of active research by a large community of mathematicians. This restriction $\pi|_{H}$ is called multiplicity-free if the only $H$-equivariant endomorphisms of the representation space of $\pi$ are the scalar multiples of the identity operator. If $H$ is compact, as it is in our setting, then this precisely means that every irreducible representation has multiplicity $1$ in $\pi|_H$.
We just mention a few results on multiplicity-free restrictions here that are particularly relevant to our approach. See for example \cite{Kobayashi08} and the references given there for more information. 

Many results about multiplicity-freeness apply to noncompact simple groups $G$ of Hermitian type. This means that $G/K$ is a Hermitian symmetric space, or equivalently, $\kk$ has nonzero centre. For such groups, $\pi$ is said to be of scalar type if the $+i$ eigenspace of the action by a fixed central element of $\kk$ on the space of $K$-finite vectors is one-dimensional. In this setting,
Kobayashi proved that  $\pi$ has multiplicity-free restriction to any subgroup $H$ such that $(G,H)$ is a symmetric pair. See \cite{Kobayashi97}, and also Theorem A in \cite{Kobayashi08}. There are many other results on multiplicity-free restrictions; two of many possible references are \cite{Kobayashi04, SunZhu12}.

Theorem \ref{thm main intro} implies a geometric sufficient condition for the restriction of $\pi$ to $K$ to be multiplicity-free: this is the case when $(G/H)_{(\eta+ \rho^K)/i}$ is a point. In fact, one can then determine explicitly which multiplicities equal $1$ and which equal $0$.
%
\begin{corollary}\label{cor intro 1}
If $(\eta+ \rho^K)/i$ is a regular value of $\Phi$ and $(G/H)_{(\eta+ \rho^K)/i}$ is a point, then $[\pi|_K:\delta] \in \{0,1\}$. The condition in Corollary \ref{cor one 2} determines precisely when the value $0$ or $1$ is taken.

If $(\eta + \rho^K)/i$ is not necessarily a regular value of $\Phi$, but $(G/H)_{(\eta+ \rho^K+ \varepsilon)/i}$ is a point for all $\varepsilon$ close enough to $0$, then we still have
$
[\pi|_K:\delta] \in \{0,1\}.
$
\end{corollary}

There is in fact a version of Corollary \ref{cor intro 1} for general admissible representations, see Corollary \ref{cor admissible}.
By applying this version, we find that the restriction to $K$ of every admissible representation is multiplicity-free in the examples where $G$ is one of the groups
\begin{itemize}
\item $\SU(p,1)$;
\item $\SO_0(p,1)$ or $\SO_0(2,2)$.
\end{itemize}
This is worked out in Section \ref{sec mult free}, see Corollary \ref{cor mult free ex}.
For $\SU(p,1)$ and $\SO_0(p,1)$, this was shown by Koornwinder \cite{Koornwinder82}. (In a related result for $\SU(p,1)$, van Dijk and Hille showed that the tensor product of a holomorphic discrete series representation and the corresponding anti-holomorphic discrete series representation decomposes multiplicity-freely; see Section 12 in \cite{vDH97}.)
For $G = \SL(2,\C)$ and $\SL(2,\R)$, all reduced spaces are points, so that all tempered representations have multiplciity free restrictions to $K$, as is well-known. We work out the case $G=\SL(2,\R)$ in detail in Subsection \ref{sec SL2}. Then we recover the well-known multiplicities of $K$-types of the tempered representations of $\SL(2,\R)$. For $\SL(2,\R)$, we show how Corollary \ref{cor intro 1} does not just imply multiplicity-freeness, but allows us to compute precisely which representations occur.

As mentioned above, for
many results on multiplicity-free restrictions, the group $G$ is assumed to be of Hermitian type. 
The groups $\SO_0(p,1)$ and $\SO_0(2,2)$ are not of Hermitian type, and can therefore not be treated via such results.

Links between multiplicity-free restrictions and the orbit method were investigated in \cite{Benson94, CG88, DV10, KN03, Nasrin10}. A key role here is played by the Corwin--Greenleaf multiplicity function $n$. For a closed subgroup $H<G$ and coadjoint orbits $\cO^H \in \kh^*/H$ and $\cO^G \in \kg^*/G$, this function takes the value
\[
n(\cO^G, \cO^H) = \# \bigl(\cO^G \cap p^{-1}(\cO^H)/H\bigr),
\]
where $p\colon \kg^*\to \kh^*$ is the restriction map. Corwin and Greenleaf  showed that this function gives multiplicities of restrictions of unitary irreducible representations if $G$ is nilpotent (see Theorem 4.8 in \cite{CG88}). Then Kirillov's orbit method classifies unitary irreducible representations as geometric quantisations of coadjoint orbits. In general, if $\pi$ is associated to $\cO^G$, then \eqref{eq QR orbit} suggests that
 the restriction $\pi|_H$  should be multiplicity-free if $n(\cO^G, \cO^H) \leq 1$ for all coadjoint orbits $\cO^H$ of $H$. Kobayashi conjectured \cite{KN03} that for groups $G$ of Hermitian type, this is the case if 
\beq{eq Nasrin intro}
\cO^{G} \cap ([\kk,\kk]+ \kp)^{\perp} \not=\emptyset.
\eeq
This conjecture was proved for $H=K$ by Nasrin \cite{Nasrin10}. 
Using Nasrin's result, we deduce the following fact from Corollary \ref{cor intro 1}.
\begin{corollary}
Suppose \eqref{eq Nasrin intro} holds. Then, under a regularity condition on  $\delta \in \hat K$, we have 
\[
[\pi|_K : \delta] \in \{0,1\},
\]
and there is a criterion for this multiplicity to equal zero or one.
\end{corollary}
See Corollary \ref{cor CG} for a precise statement. 

We conjecture the condition for multiplicity-free restrictions in 
Corollary \ref{cor intro 1} to be necessary, as well as sufficient.
\begin{conjecture}
Let $H<G$ be a $\theta$-stable Cartan subgroup. Suppose that every tempered representation induced from the cuspidal parabolic subgroup corresponding to $H$ restricts multiplicity-freely to $K$. Then all reduced spaces for all maps $\Phi\colon G/H \to \kk^*$ corresponding to those representations are points.
\end{conjecture}
Evidence for this conjecture is given under 
Conjecture \ref{conj mult free}.

\subsection{Ingredients of the proof}

The proof of Theorem \ref{thm main intro} is based on three ingredients.
\begin{enumerate}
\item A realisation of $\pi|_K$ as a $K$-equivariant index of a deformed Dirac operator on $G/H$. This was done in  Theorem 3.11 in \cite{HSY1}. That result involves index theory of deformed Dirac operators developed by Braverman \cite{Braverman02}.
\item A general quantisation commutes with reduction result for noncompact $\Spinc$-manifolds. This is Theorem 3.10 in \cite{HS16}. For compact $\Spinc$-manifolds, this was proved by Paradan--Vergne \cite{Paradan15, Paradan14, Paradan14CRAS}. For noncompact symplectic manifolds, the analogous result was proved by Ma--Zhang \cite{Zhang14}, after a conjecture by Vergne \cite{Vergne06}. See also \cite{Paradan11}.
\item One needs to show that the second ingredient can be applied to the first, by using the freedom one has
 in the deformation of the Dirac operator on $G/H$ to choose the particular  deformation that yields the desired result.
 This requires some work, and occupies a large part of this paper.
\end{enumerate}

\subsection*{Acknowledgements}

The authors are grateful to Maxim Braverman, Paul-\'Emile Paradan and David Vogan for their hospitality and inspiring discussions at various stages. The authors thank Mich\`ele Vergne for useful advice, and input on Conjecture \ref{conj mult free}.

The first author was partially supported by the European Union, through Marie Curie fellowship PIOF-GA-2011-299300. He thanks Dartmouth College for funding a visit there in 2016.

\subsection*{Notation}

The Lie algebra of a Lie group is denoted by the corresponding lower case Gothic letter. We denote complexifications by superscripts $\C$. The unitary dual of a group $H$ will be denoted by $\hat H$. If $H$ is an abelian Lie group and $\xi \in \kh^*$ satisfies the appropriate integrality condition, then we write $\C_{\xi}$ for the one-dimensional representation of $H$ with weight $\xi$. 

{In Subsections \ref{sec ind red}, \ref{sec QR=0} and \ref{sec sing}, the letter $M$ denotes a manifold. In the rest of this paper, it denotes a subgroup of $G$.} 


\section{The multiplicity formula}\label{sec mult form}

The main result of this paper is a multiplicity formula for $K$-types of tempered representations, Theorem \ref{thm mult form}, and its extension, Corollary \ref{cor mult form K'}. This is a geometric formula in terms of indices on \emph{reduced spaces} for the action by a maximal compact subgroup on a homogeneous space of the group in question.

\subsection{Indices on reduced spaces} \label{sec ind red}

Let $M$ be a complete Riemannian manifold, on which a compact Lie group $K$ acts isometrically. Let $J$ be a $K$-invariant almost complex structure on $M$. We write $\Bigwedge_J TM$ for the complex exterior algebra bundle of $TM$, viewed as a complex bundle via $J$.
Let $L \to M$ be a Hermitian, $K$-equivariant line bundle. The vector bundle
\beq{eq spinor bundle}
\Bigwedge_J TM \otimes L \to M
\eeq
is the spinor bundle of the $\Spinc$-structure on $M$ defined by $J$ and $L$, see e.g.\ Proposition D.50 in \cite{Guillemin98} or page 395 of \cite{Lawson89}. In this paper, we will only work with $\Spinc$-structures induced by almost complex structures and line bundles as in this case.

The determinant line bundle associated to the $\Spinc$-structure with spinor bundle \eqref{eq spinor bundle} is
\[
L_{\det} = \Bigwedge^{\dim(M)/2}_J TM \otimes L^{\otimes 2} \to M.
\]
Let $\nabla$ be a $K$-invariant, Hermitian connection on $L_{\det}$. The corresponding \emph{moment map} is the map $\Phi\colon M \to \kk^*$ such that for all $X \in \kk$,
\beq{eq def moment}
2i\langle \Phi, X\rangle = \calL_X - \nabla_{X^M}.
\eeq
Here $\langle \Phi, X\rangle \in C^{\infty}(M)$ is the pairing of $\Phi$ and $X$, $\calL_X$ is the Lie derivative with respect to $X$ of smooth sections of $L_{\det}$, and $X^M$ is the vector field on $M$ induced by $X$; our sign convention is that for $m \in M$,
\[
X^M(m) = \ddt \exp(-tX)m.
\]   
{The origin of the term `moment map' is that, by Kostant's formula,  $\Phi$ is a moment map in the symplectic sense if the curvature of $\nabla$ is $-i$ times a symplectic form on $M$.}

If $\xi \in \kk^*$, then the \emph{reduced space} at $\xi$ is the space
\beq{eq def reduced sp}
M_{\xi} := \Phi^{-1}(\xi)/K_{\xi},
\eeq
where $K_{\xi}$ is the stabiliser of $\xi$ with respect to the coadjoint action. If $\xi$ is a regular value of $\Phi$, then $K_{\xi}$ acts on the smooth submanifold $\Phi^{-1}(\xi) \subset M$ with finite stabilisers. Then $M_{\xi}$ is an orbifold. In our setting, the map $\Phi$ will be proper, so that $M_{\xi}$ is compact. We will express multiplicities of $K$-types of tempered representations as indices of Dirac operators on reduced spaces. For reduced spaces at regular values of the moment map, these are indices in the orbifold sense. For reduced spaces at singular values, one applies a small shift to a nearby regular value, see Definition \ref{def quant red} below.

The indices on reduced spaces that we will use were constructed in Subsections 5.1 and 5.2 of \cite{Paradan14}, for general $\Spinc$-structures. We review this construction here, for $\Spinc$-structures induced by almost complex structures and line bundles as above. The construction is done in three steps. First, one realises a given reduced space as a reduced space for an action by a torus. For actions by tori, indices on reduced spaces at regular values of the moment map can be defined directly. For singular values, one applies a shift to a nearby regular value.

We suppose from now on that the action by $K$ on $M$ has abelian stabilisers. (This is true in our application of what follows.)

Let $T<K$ be a maximal torus. Fix an open Weyl chamber $C \subset \kt$, and let $\rho^K$ be half the sum of the corresponding positive roots. 
Let $\xi \in \kt^*$ be dominant with respect to $C$. Then $\xi + \rho^K/i \in C$. We will always  identify $\kk \cong \kk^*$ via the inner product equal to minus the Killing form. Let $Y\subset M$ be a connected component of $\Phi^{-1}(C)$. Consider the map
\[
\Phi_Y := \Phi|_{Y} - \rho^K/i \colon Y \to \kt^*.
\]
Set
\[
Y_{\xi} := \Phi_Y^{-1}(\xi)/T.
\]
Let $q\colon \Phi_Y^{-1}(\xi)\to Y_{\xi}$ be the quotient map. Let $\kt_Y \subset \kt$ be the generic (i.e.\ minimal) stabiliser of the infinitesimal action by $\kt$ on $Y$.  The image of $\Phi_Y$ is contained in an affine subspace $I(Y) \subset \kt^*$ parallel to the annihilator of $\kt_Y$.
\begin{lemma}[Paradan--Vergne]\label{lem Spinc red}
If $\xi$ is a regular value of $\Phi_Y\colon Y \to I(Y)$, then $Y_{\xi}$ is an orbifold, and for every integral element $\eta \in I(Y)$, 
there is an orbifold $\Spinc$-structure on $Y_{\xi}$, with spinor bundle $\calS^{\eta}_{Y, \xi} \to Y_{\xi}$ determined by
\[
\bigl(\Bigwedge_J TM \otimes L\bigr)|_{\Phi_Y^{-1}(\xi)} = q^*\calS^{\eta}_{Y, \xi} \otimes_{\C} \Bigwedge_{\C}\kk/\kt \otimes_{\C} \bigl( \Bigwedge_{\C} (\kt/\kt_Y\otimes_{\R} \C )\bigr) \otimes_{\C} \C_{\eta}.
\]
Here $\kk/\kt$ is viewed as a complex vector space isomorphic to the sum of the positive root spaces corresponding to $C$, and $\Bigwedge_{\C}$ denotes the exterior algebra of complex vector spaces.
\end{lemma}
This is Lemma 5.2 in \cite{Paradan14}.

Suppose that $\Phi$ is a proper map. Then $Y_{\xi}$ is compact.
In the setting of Lemma \ref{lem Spinc red}, we write
\[
\indx(\calS^{\eta}_{Y, \xi}) \quad \in \Z
\]
for the orbifold index \cite{Kawasaki81} of a $\Spinc$-Dirac operator on the bundle $\calS^{\eta}_{Y, \xi}$. This can be evaluated in terms of characteristic classes on $Y_{\xi}$ via Kawasaki's index theorem, see formula (7) in \cite{Kawasaki81}.

\begin{theorem}[Paradan--Vergne]\label{thm indx red}
The integer $\indx(\calS^{\xi}_{Y, \xi + \varepsilon})$ is independent of $\varepsilon \in I(Y)$ such that $\xi + \varepsilon$ is a regular value of $\Phi_Y\colon Y \to I(Y)$, for $\varepsilon$ small enough.
\end{theorem}
This result is Theorem 5.4 in \cite{Paradan14}. It allows us to define
\[
\indx(\calS^{\xi}_{Y, \xi}) := \indx(\calS^{\xi}_{Y, \xi + \varepsilon}),
\]
for $\varepsilon$ as in Theorem \ref{thm indx red}.

Finally, we have
\[
M_{\xi + \rho^K/i} = \coprod_{Y} Y_{\xi},
\]
where $Y$ runs over the connected components of $\Phi^{-1}(C)$. 
\begin{definition} \label{def quant red}
The \emph{index of the $\Spinc$-Dirac operator on the reduced space $M_{\xi + \rho^K/i}$} is the integer
\[
\indx(M_{\xi + \rho^K/i}) = \sum_{Y} \indx(\calS^{\xi}_{Y, \xi}),
\]
where $Y$ runs over the connected components of $\Phi^{-1}(C)$. If $M_{\xi + \rho^K/i} = \emptyset$, then we set $\indx(M_{\xi + \rho^K/i}) = 0$.
\end{definition}
Such an index on a reduced space may be viewed as the \emph{$\Spinc$-quantisation} of that space; see Definition 5.5 in \cite{Paradan14}.

\subsection{Tempered representations, almost complex structures and moment maps} \label{sec tempered reps}

Let $G$ be a connected, linear, real reductive Lie group with compact centre. Let $K<G$ be maximal compact, and let $\theta$ be a compatible Cartan involution. A tempered representation of $G$ is an irreducible unitary representation whose $K$-finite matrix coefficients are in $L^{2+\varepsilon}(G)$ for all $\varepsilon > 0$. These are the representations that occur in the Plancherel decomposition of $L^2(G)$.
Let $\pi$ be a tempered representation of $G$.

Tempered representations were classified by Knapp and Zuckerman. See \cite{KZ1, KZ2, KZ3} or Chapter XIV in \cite{KnappBook} for details, or Subsection 2.3 of \cite{HSY1} for a brief overview of the parts relevant to us here. In this classification one parametrises $\pi$ as follows. Let $P = MAN$ be the Langlands decomposition of a cuspidal parabolic subgroup $P<G$. Let $H<G$ be the $\theta$-stable Cartan subgroup with noncompact part $A$. Write $\kh = \kt_M \oplus \ka$, with $\kt_M \subset \km$. Then $\kt_M$ is a Cartan subalgebra of $\km$. Set $T_M := \exp(\kt_M)$. 
Recall that we use minus the Killing form, which we denote by $(\relbar, \relbar)$,   to identify $\kk^* \cong \kk$.
Let $\lambda \in i\kt_M^*$, and let $R^+_M$ be a system of positive roots for $(\km^{\C}, \kt_M^{\C})$ such that for all $\alpha \in R^+_M$, we have $(\alpha, \lambda) \geq 0$. Let $\rho^M$ be half the sum of the elements of $R^+_M$. Suppose that $\lambda - \rho^M$ is integral. Let $Z_M$ be the centre of $M$. Let $\chi_M$ be a one-dimensional representation of $Z_M$ such that
\[
\chi_M|_{T_M \cap Z_M} = \C_{\lambda - \rho^M}|_{T_M \cap Z_M}.
\]
Then we have the well-defined representation
$
\C_{\lambda - \rho^M} \boxtimes \chi_M
$
of $H_M = T_MZ_M$.
One has the discrete series or limit of discrete series representation $\pi^M_{\lambda, R^+_M, \chi_M}$ associated to these data; see page 397 of \cite{KZ1}. (For singular $\lambda$, $\pi^M_{\lambda, R^+_M, \chi_M}$ is a limit of discrete series representation if it is nonzero.)
Let $\nu \in i\ka^*$. For suitable $\lambda$, $R^+_M$ and $\chi_M$  as above, we have
\beq{eq pi ind}
\pi = \Ind_{MAN}^G(\pi^M_{\lambda, R^+_M, \chi_M} \otimes e^{\nu} \otimes 1_N).
\eeq
This is Corollary 8.8 in \cite{KZ1}.

We will use the $K$-invariant almost complex structure $J$ on $G/H$ defined in Subsection 3.4 of \cite{HSY1}. This was defined via the decomposition
\beq{eq decomp gh}
\kg/\kh \cong \km/\kt_M \oplus \kn^- \oplus \kn^+,
\eeq
where $\kn^+ = \kn$ and $\kn^- = \theta \kn^+$. On $\km/\kt_M$ we have the complex structure $J_{\km/\kt_M}$ such that, as complex vector spaces,
\[
\km/\kt_M = \bigoplus_{\alpha \in R^+_M} \km^{\C}_{\alpha}.
\]
Let $\Sigma$ be the set of nonzero weights of the adjoint action by $\ka$ on $\kg$. For $\beta \in \Sigma$, let $\kg_{\beta} \subset \kg$ be the corresponding weight space. Let $\Sigma^+ \subset \Sigma$ be the set of positive weights such that
\[
\kn = \bigoplus_{\beta \in \Sigma^+} \kg_{\beta}.
\]
Let $\zeta \in \ka$ be an element for which $\langle \beta, \zeta \rangle > 0$ for all $\beta \in \Sigma^+$. Then the map
\[
J_{\zeta}:= \theta |\ad(\zeta)|^{-1} \ad(\zeta)\colon \kn^- \oplus \kn^+ \to \kn^- \oplus \kn^+
\]
is an $H_M$-invariant complex structure (see Lemma 3.9 in \cite{HSY1}). Let $J_{\kg/\kh}$ be the complex structure on $\kg/\kh$ defined by $J_{\km/\kt_M}$ and $J_{\zeta}$ via \eqref{eq decomp gh}. Then $J$ is the $K$-invariant almost complex structure on $G/H$ such that for all $k \in K$, $X \in \ks_M$ and $Y \in \kn$, the following diagram commutes:
\[
\xymatrix{
T_{k\exp(X)\exp(Y)H}G/H \ar[r]^-{J} & T_{k\exp(X)\exp(Y)H}G/H  \\
T_{eH}G/H = \kg/\kh \ar[u]^-{T_{eH}k\exp(X)\exp(Y)} \ar[r]_-{J_{\kg/\kh}} & \kg/\kh  = T_{eH}G/H \ar[u]_-{T_{eH}k\exp(X)\exp(Y)}.
}
\]
(See Lemma 3.10 in \cite{HSY1}.)

Consider the line bundle
\[
L_{\lambda - \rho^M, \chi_M} := G\times_H \C_{\lambda - \rho^M} \boxtimes \chi_M \to G/H
\]
(where we extend $\lambda - \rho^M \in i\kt_M^*$ to $\kh$ by setting it equal to zero on $\ka$). The vector bundle
\beq{eq spinor GH}
\Bigwedge_J T(G/H) \otimes L_{\lambda - \rho^M, \chi_M} \to G/H
\eeq
is a spinor bundle of the form \eqref{eq spinor bundle}.

The positive systems $R^+_M$ and $\Sigma^+$ determine a set $R^+_G$ of positive roots $\alpha$ of $(\kg^{\C}, \kh^{\C})$ that satisfy
\beq{eq RG}
(\alpha|_{\ka} = 0 \text{ and } \alpha|_{\kt_M} \in R^+_M) \quad \text{or} \quad (\alpha|_{\ka} \not=0 \text{ and } \alpha|_{\ka} \in \Sigma^+).
\eeq
Let $\rho^G$ be half the sum of the elements of $R^+_G$. Set
\beq{eq def xi}
\xi := (\lambda + \rho^{G,M})/i \quad \in \kt_M^*, 
\eeq
where
\beq{eq def rhoGM}
\rho^{G,M} := \rho^G|_{\kt_M} - \rho^M \quad \in i\kt_M^*.
\eeq
\begin{proposition} \label{prop moment map}
There is a $K$-invariant, Hermitian connection on the determinant line bundle corresponding to \eqref{eq spinor GH} whose moment map $\Phi\colon G/H \to \kk^*$ is given by
\[
\Phi(gH) = (\Ad^*(g)(\xi + \zeta))|_{\kk}.
\]
\end{proposition}

\subsection{The main result} \label{sec main result}

Fix a set of positive roots of $(\kk^{\C}, \kt^{\C})$ compatible with $R^+_M$. Let $\rho^K$ be half the sum of these positive roots.
Let $\delta \in \hat K$, and let $\eta \in i\kt^*$ be its highest weight.
If $\xi \in \kt_M^*$ is regular, it has positive inner products with all roots in $R^+_M$. Then we can and will choose $\zeta$ such that  $\xi + \zeta$ is regular for the roots of $(\kg^{\C}, \kh^{\C})$. 
In this case, let 
\[
(G/H)_{(\eta+\rho^K)/i} = \Phi^{-1}((\eta+\rho^K)/i)/T
\]
be the reduced space at $(\eta+\rho^K)/i$, as in \eqref{eq def reduced sp}, for the moment map
$\Phi$ of Proposition \ref{prop moment map}. Let 
\[
\indx((G/H)_{(\eta+\rho^K)/i}) \quad \in \Z
\]
be the index of the $\Spinc$-Dirac operator on this space, as in Definition \ref{def quant red}. Recall that
\[
\pi = \Ind_{MAN}^G(\pi^M_{\lambda, R^+_M, \chi_M} \otimes e^{\nu} \otimes 1_N).
\]
We set $K_M := K \cap M$. 
\begin{theorem}[Multiplicity formula; 
regular case]\label{thm mult form reg}
Suppose that $\xi \in \kt_{M}^*$ is regular.
For all $\delta \in \hat K$, with highest weight $\eta$, the multiplicity of $\delta$ in  $\pi|_K$ is
\[
[\pi|_K: \delta] = (-1)^{\dim(M/K_M)/2} \indx((G/H)_{(\eta+\rho^K)/i}). 
\]
\end{theorem}
Note that for $\xi$ to be regular, it is sufficient that $\lambda$ is regular; i.e.\  
$\pi^M_{\lambda, R^+_M, \chi_M}$ belongs to the discrete series of $M$.

If $\xi \in \kt_M^*$ is singular, 
choose any $\tilde \xi \in \kt_M^*$ with positive inner products with the positive roots in $R^+_M$, and choose $\zeta$ such that $\tilde \xi +  \zeta$ is regular. Define the map $\psi\colon G/H \to \kk^*$ by 
\[
\psi(gH) = (\Ad^*(g)(\tilde \xi + \zeta))|_{\kk},
\]
for $g \in G$.
Let $v^{\psi}$ be the vector field on $G/H$ defined by
\[
v^{\psi}(gH) = \ddt \exp(-t\psi(gH))gH,
\]
for all $g \in G$. Next, we choose a nonnegative function $\tau \in C^{\infty}(G/H)^K$ that grows fast enough, as in  Lemma \ref{lem taming moment} below. That lemma implies that
the map $\Phi^{\tau}\colon G/H \to \kk$ given by
\beq{eq mu tau GH}
\langle \Phi^{\tau}, X\rangle = \langle \Phi, X\rangle + \tau \cdot (v^{\psi}, X^{G/H}),
\eeq
for $X \in \kk$, is a proper moment map. In this case, we set
\[
(G/H)_{(\eta+\rho^K)/i} = (\Phi^{\tau})^{-1}((\eta+\rho^K)/i)/T.
\]
Again, 
let 
\[
\indx((G/H)_{(\eta+\rho^K)/i}) \quad \in \Z
\]
be the index of the $\Spinc$-Dirac operator on this space. 
\begin{theorem}[Multiplicity formula; 
singular case]\label{thm mult form sing}
Suppose $\xi \in \kt_M^*$ is singular. 
For all $\delta \in \hat K$, with highest weight $\eta$, the multiplicity of $\delta$ in  $\pi|_K$ is
\[
[\pi|_K: \delta] = (-1)^{\dim(M/K_M)/2} \indx((G/H)_{(\eta+\rho^K)/i}). 
\]
\end{theorem}

Combining Theorems \ref{thm mult form reg} and \ref{thm mult form sing}, we obtain the main result of this paper.
\begin{theorem}[Multiplicity formula for $K$-types of tempered representations]\label{thm mult form}
For any tempered representation $\pi$ of $G$, and all $\delta \in \hat K$, with highest weight $\eta$, the multiplicity of $\delta$ in  $\pi|_K$ is
\[
[\pi|_K: \delta] = (-1)^{\dim(M/K_M)/2} \indx((G/H)_{(\eta+\rho^K)/i}). 
\]
In other words,
\[
\pi|_K = (-1)^{\dim(M/K_M)/2} \bigoplus_{\delta \in \hat K} \indx((G/H)_{(\eta+\rho^K)/i}) \delta.
\]
\end{theorem}

If $\pi$ belongs to the discrete series, then this multiplicity formula is Theorem 2.5 in \cite{Paradan03}. The absence of the sign $(-1)^{\dim(M/K_M)/2} = (-1)^{\dim(G/K)/2}$ in that result is due to a different definition of reduced spaces and the relevant indices on them. Our proof of Theorem \ref{thm mult form} is based on a generalisation of the methods in \cite{Paradan03}, combined with Braverman's index theory described in Subsection \ref{sec QR=0}. Via a completely different method, Duflo and Vergne \cite{Duflo11} proved Theorem \ref{thm mult form} in the regular case, where $\pi^M_{\lambda, R^+_M, \chi_M}$ belongs to the discrete series. 
Duflo and Vergne used Kirillov's character formula, proved by Rossman \cite{Rossmann78}; see also \cite{Vergne79}. This formula is based on deep results of Harish-Chandra and others. Our approach uses a geometric realisation of $\pi|_K$ instead, and in addition covers the singular case.

Theorem \ref{thm mult form} can in fact be generalised to more general compact subgroups of $G$, using a functoriality result by Paradan \cite{Paradan17}. Let $K<G$ now be any compact  subgroup, not necessarily maximal. Let $K'<G$ be a maximal compact subgroup containing $K$.

Let $\Phi^{\tau}$ be as above for the action by $K'$ on $G/H$, where we take $\tau=0$, so $\Phi^{\tau} = \Phi$, in the regular case.  Suppose that the composition
\beq{eq Phi K}
\Phi_{K}\colon G/H \xrightarrow{\Phi^{\tau}} (\kk')^* \xrightarrow{p} \kk^*
\eeq
is proper, where $p$ is the restriction map. (This is true if $K'=K$ as in Theorem \ref{thm mult form}.) The multiplicity formula by Duflo and Vergne \cite{Duflo11}  for tempered representation induced from the discrete series holds for restrictions to compact subgroups $K$ with this property. Let $T$ be a maximal torus of $K$. For $\xi \in \kt^*$, we write
\[
(G/H)^{K}_{\xi} := \Phi_{K}^{-1}(\xi)/K_{\xi}.
\]
This allows us to state the most general multiplicity formula in this paper.
\begin{corollary} \label{cor mult form K'}
The restriction of $\pi$ to $K$ is admissible, and we have
\[
\pi|_{K} = (-1)^{\dim(M/K'_M)/2} \bigoplus_{\delta \in \hat K} \indx((G/H)^{K}_{(\eta'+\rho^{K})/i} ) \delta.
\]
Here, as before, $\eta$ is the highest weight of an irreducible representation $\delta \in \hat K$.
\end{corollary}
\begin{remark} \label{rem standard reps}
In fact, Corollary \ref{cor mult form K'} applies to every representation $\pi$ of the form \eqref{eq pi ind} with $\nu \in (\ka^{\C})^*$ possibly non-imaginary, i.e.\ to every \emph{standard representation}. This includes the tempered representations by Corollary 8.8 in \cite{KZ1}.
\end{remark}

This multiplicity formula is an instance of the $\Spinc$-version of the \emph{quantisation commutes with reduction} principle. Indeed, we will see in Theorem \ref{thm realise}, which is the main result in \cite{HSY1}, that one can view $\pi|_K$ as the geometric quantisation in the $\Spinc$-sense of the action by $K$ on $G/H$, with the given almost complex structure and line bundle. If the infinitesimal character $\chi$ of $\pi$ is regular, then $G/H$ is isomorphic to the coadjoint orbit $\Ad^*(G)(\chi + \rho^{G,M})$ as $K$-equivariant $\Spinc$-manifolds. Now $\Phi$ is the natural projection of this orbit onto $\kk^*$. See Section 3.6 in \cite{HSY1} for this relation with the orbit method.  
This $\Phi$ is the moment map for the natural symplectic form on this orbit. Nevertheless, one needs a $\Spinc$-version of the quantisation commutes with reduction principle (Theorem 3.10 in \cite{HS16}; see Theorem \ref{thm QR=0} below) rather than the symplectic version (Theorem 0.1 in \cite{Zhang14}, see also Theorem 1.4 in \cite{Paradan11}). This is because the almost complex structure $J$ is not compatible with the Kostant--Kirillov symplectic form on the coadjoint orbit $\Ad^*(G)(\chi + \rho^{G,M})$; also, $L_{\lambda - \rho^M, \chi_M} $ is not a prequantum line bundle for this symplectic form. See Subsection 1.5 in \cite{Paradan03}. The bundle $L_{\lambda - \rho^M, \chi_M} $ is a $K$-equivariant prequantum line bundle for the coadjoint orbit $\Ad^*(G)2(\chi+\rho^{G.M})$, however; i.e.\ the spinor bundle \eqref{eq spinor GH} is a $\Spinc$-prequantisation of the orbit $\Ad^*(G)(\chi+\rho^{G.M})$.
See Remark \ref{rem preq line}.
 In the compact case, the $\Spinc$-version of the quantisation commutes with reduction principle was proved by Paradan and Vergne \cite{Paradan15, Paradan14, Paradan14CRAS}. 

In the orbit method, representations with singular parameters correspond to nilpotent orbits.
If $\chi$ is singular, then we use the manifold $G/H$ rather than such a nilpotent orbit. Through this desingularisation, the link with quantising nilpotent orbits is absent in our approach, but this approach does allow us to obtain the multiplicity formula in Theorem \ref{thm mult form} and Corollary \ref{cor mult form K'}.


Theorem \ref{thm mult form} and Corollary \ref{cor mult form K'} allow us the deduce properties of the behaviour of the $K$-type multiplicities of $\pi$ from the geometry of the coadjoint orbit $\Ad^*(G)(\xi + \zeta)$ if $\xi$ is regular.  In general, such properties can be deduced from the geometry of the map $\Phi\colon G/H \to \kk^*$. An immediate consequence is the following fact
%
%
%
%
%
%
about the support of the multiplicity function of the $K$-types of $\pi$. By the relative interior or relative boundary of a subset of the affine space $I(Y)$ parallel to the annihilator of $\kt_Y$ containing the image of $\Phi_Y$, we mean the interior or boundary as a subset of $I(Y)$.
\begin{corollary} \label{cor zero}
Let $K$ be as in Corollary \ref{cor mult form K'}.
All $K$-types of $\pi$ have highest weights in the relative interior of $i\Phi(G/H)  \cap i\kt^* - \rho^K$.
\end{corollary}
\begin{proof}
Let $\delta \in \hat K$ have highest weight $\eta$. If $(\eta + \rho^K)/i$ is not in the image of $\Phi$, then 
Corollary \ref{cor mult form K'}
implies that  the multiplicity $[\pi|_K:\delta]$ is zero because $(G/H)_{(\eta+\rho^K)/i}$ is empty, and so are reduced spaces at elements close enough to $(\eta+\rho^K)/i$. If $(\eta + \rho^K)/i$ lies on the relative boundary of the image of $\Phi$, then $[\pi|_K:\delta]$ is zero because the reduced space at some element close to $(\eta + \rho^K)/i$ is empty. See 
the comment below Definition 5.5 in \cite{Paradan14}.
\end{proof} 

\begin{remark}
In the regular case, the map $\Phi$ is a moment map in the symplectic sense. So then the set $i\Phi(G/H)  \cap i\kt^* - \rho^K$ containing the support of the multiplicity function is a convex polytope. This polytope is noncompact if $G$ is; i.e.\ it is the intersection of a collection of half-spaces.
\end{remark}

\begin{remark}
Even in the case of the discrete series, it is nontrivial to determine the support of the multiplicity function from Blattner's formula. This is because of cancellations occurring in that formula. 
\end{remark}

Applications of Theorem \ref{thm mult form} and Corollary \ref{cor mult form K'} to multicplicity-free restrictions are discussed in Section \ref{sec mult free}.


\section{Ingredients of the proof}

\subsection{Quantisation commutes with reduction} \label{sec QR=0}

Consider the setting of Subsection \ref{sec ind red}. Let $\psi\colon M \to \kk$ be a smooth, $K$-equivariant map. It 
 induces a vector field $v^{\psi}$, given by
\[
v^{\psi}(m) = \ddt \exp(-t\psi(m))m,
\]
for all $m \in M$. The map $\psi$ is called \emph{taming} if the set of zeroes of $v^{\psi}$ is compact.

%

The Clifford action $c$ by $TM$ on $\Bigwedge_JTM$ is given by
\[
c(v)x = v\wedge x - v^* \lrcorner x
\]
for $m \in M$, $v \in T_mM$ and $x \in \Bigwedge_JT_mM$. Here $v^* \in T^*M$ is dual to $v$ with respect to the Hermitian metric defined by the Riemannian metric and $J$, and $\lrcorner$ denotes contraction. Let $\tilde \nabla$ be a $K$-invariant, Hermitian connection on $\Bigwedge_JTM \otimes L$, such that for all vector fields $v,w$ on $M$,
\[
[\tilde \nabla_v, c(w)] = c(\nabla^{TM}_v w),
\]
where $\nabla^{TM}$ is the Levi--Civita connection. Such a connection always exists; one is induced by the connections $\nabla$ on $L_{\det}$ and $\nabla^{TM}$ on $TM$, see e.g.\ Proposition D.11 in \cite{Lawson89}. After we identify $T^*M \cong TM$ via the Riemannian metric, the Clifford action $c$ induces a map 
\[
c\colon T^*M \otimes  \Bigwedge_JTM \otimes L \to \Bigwedge_JTM \otimes L.
\]
This allows us to define the \emph{Dirac operator} $D$ as the composition
\[
D\colon \Gamma^{\infty}(\Bigwedge_JTM \otimes L) \xrightarrow{\tilde \nabla}  \Gamma^{\infty}(T^*M \otimes \Bigwedge_JTM \otimes L) \xrightarrow{c} \Gamma^{\infty}(\Bigwedge_JTM \otimes L).
\]

Let $f \in C^{\infty}(M)^K$ be nonnegative. 
The \emph{Dirac operator deformed by $f\psi$} is the operator
\[
D -ifc(v^{\psi})
\]
on the space $\Gamma_{L^2}^{\infty}(\Bigwedge_J TM \otimes L)$ of square-integrable smooth sections of $\Bigwedge_J TM \otimes L$.
For a nonnegative function $\chi \in C^{\infty}(M)^K$, we say that the function $f$ is \emph{$\chi$-admissible} if, outside a compact set,
\[
\frac{f^2}{\|df\| + f + 1} \geq \chi.
\]
For any such function $\chi$, there exist $\chi$-admissible functions, see Lemma 3.10 in \cite{HS1}. Braverman's index theory \cite{Braverman02} for deformed Dirac operators is based on the following result.
\begin{theorem}[Braverman] \label{thm index}
If $\psi$ is taming, then there is a nonnegative function $\chi \in C^{\infty}(M)^K$, such that for all $\chi$-admissible functions $f$, and all irreducible representations $\delta$ of $K$, the multiplicities $m_{\delta}^{+}$ and $m_{\delta}^-$ of $\delta$ in the kernel of $D -ifc(v^{\psi})$ restricted to even and odd degree forms, respectively, is finite. The difference $m_{\delta}^+ - m_{\delta}^-$ is independent of $f$ and $\nabla$.
\end{theorem}
See Theorem 2.9 in \cite{Braverman02} for a more general result.

We write $\hat R(K)$ for the abelian group
\[
\hat R(K) = \Bigl\{\bigoplus_{\delta \in \hat K} m_{\delta} \delta; m_{\delta} \in \Z \Bigr\}.
\]
I.e.\ $\hat R(K)$ contains formal differences of possibly infinite-dimensional representations of $K$, in which all irreducible representations have finite multiplicities. 
\begin{definition}\label{def index}
In the setting of Theorem \ref{thm index}, the \emph{equivariant index} of the pair $(\Bigwedge_J TM \otimes L, \psi)$ is
\[
\indx_K(\Bigwedge_J TM \otimes L, \psi) = \displaystyle{\bigoplus}_{\delta \in \hat K} (m_{\delta}^+ - m_{\delta}^-)\delta \quad \in \hat R(K).
\]
\end{definition}

A property of this index is invariance under homotopies of taming maps. Two taming maps $\psi_0, \psi_1\colon M \to \kk$ are \emph{homotopic} if there is a taming map $\psi\colon M \times [0,1] \to \kk$ such that for all $m \in M$, we have $\psi(m,t) = \psi_0(m)$ if $t \in {[0,1/3[}$, and $\psi(m,t) = \psi_1(m)$  if $t \in {]2/3, 1]}$. 
\begin{theorem}[Braverman]\label{thm htp invar}
If $\psi_0$ and $\psi_1$ are homotopic taming maps, then
\[
\indx_K(\Bigwedge_J TM \otimes L, \psi_0) =\indx_K(\Bigwedge_J TM \otimes L, \psi_1). 
\]
\end{theorem}
This is a special case of cobordism invariance of the index, Theorem 3.7 in \cite{Braverman02}.

In \cite{HS16}, it was proved that the index of Definition \ref{def index} satisfies a $\Spinc$-version of the \emph{quantisation commutes with reduction} principle of Guillemin and Sternberg \cite{Guillemin82}. This followed results for compact symplectic manifolds,  \cite{Meinrenken98,Meinrenken99}, see also \cite{Paradan01,Zhang98}; for noncompact symplectic manifolds  \cite{Zhang14}, see also  \cite{Paradan11}; and for compact $\Spinc$-manifolds \cite{Paradan14}, see also \cite{Paradan15,Paradan14CRAS}. The interpretation of the $K$-equivariant index of a Dirac operator deformed by a vector field such as $v^{\Phi}$ as a geometric quantisation goes back to \cite{HS16, Paradan03, Paradan11, Zhang14,  Vergne06}.
\begin{theorem} \label{thm QR=0}
In the setting of Theorem \ref{thm index}, 
take $\psi = \Phi$, the moment map of a $K$-invariant, Hermitian connection on $L_{\det}$.
Suppose that $\Phi$ is taming and proper, and that the generic stabiliser of the action by $K$ on $M$ is abelian. 
Let $\eta$ be the highest weight of $\delta$.
Then 
\[
m_{\delta}^+ - m_{\delta}^- = \indx(M_{(\eta+\rho^K)/i}).
\]
I.e.\ 
\[
\indx_K(\Bigwedge_J TM \otimes L, \Phi) = \displaystyle{\bigoplus}_{\delta \in \hat K} \indx(M_{(\eta+\rho^K)/i}) \delta.
\]
\end{theorem}
This is a special case of Theorem 3.10 in \cite{HS16}. In that theorem it was not assumed that $\Phi$ is taming,  that the generic stabiliser is abelian, or that the $\Spinc$-structure is induced by an almost complex structure and a line bundle.

\subsection{A realisation of tempered representations restricted to $K$}

As in Subsection \ref{sec tempered reps}, let $\pi$ be a tempered representation of $G$, and write
\[
\pi = \Ind_{MAN}^G(\pi^M_{\lambda, R^+_M, \chi_M} \otimes e^{\nu} \otimes 1_N)
\]
as in \eqref{eq pi ind}. Let $H$ be the corresponding Cartan subgroup. In \cite{HSY1}, we realised the restriction of $\pi$ to $K$ as an equivariant index in the sense of Definition \ref{def index} of a deformed Dirac operator on $G/H$. We briefly review the construction here.

Consider the spinor bundle \eqref{eq spinor GH}, and the map $\Phi$ of Proposition \ref{prop moment map}, but now  for any elements
 $\xi \in \kt_M^*$ and $\zeta \in \ka^*$ such that $(\alpha, i\xi) > 0$ for all $\alpha \in R^+_M$, and $\xi + \zeta \in \kh^*$ is regular for the roots of $(\kg^{\C}, \kh^{\C})$. Then the map $\Phi$ is  taming by Proposition 2.1 in \cite{Paradan99}. 
\begin{theorem}\label{thm realise} 
We have
\[
 \pi|_K = (-1)^{\dim(M/K_M)/2}
\indx_K\bigl(\Bigwedge_{J} T(G/H) \otimes L_{\lambda- \rho^M, \chi_M}, \Phi \bigr). 
\]
\end{theorem}
This is Theorem 3.10 in \cite{HSY1}. It is the last ingredient of the proof of Theorem \ref{thm mult form}. Theorem \ref{thm realise} in fact applies more generally to every standard representation $\pi$; see Remark 3.12 in \cite{HSY1}.

\medskip
\noindent
\emph{Proof of Theorem \ref{thm mult form}.}
Let $\delta \in \hat K$. Let $\xi$ be as in \eqref{eq def xi}. First suppose that $\xi$ is regular, and choose $\zeta$ such that 
 $\xi + \zeta \in \kh^*$ is regular for the roots of $(\kg^{\C}, \kh^{\C})$. Then Theorem \ref{thm realise} states that
\[
[\pi|_K: \delta] =  (-1)^{\dim(M/K_M)/2}\bigl[
\indx_K\bigl(\Bigwedge_{J} T(G/H) \otimes L_{\lambda- \rho^M, \chi_M}, \Phi \bigr):\delta \bigr].
\]
By Proposition \ref{prop moment map}, the map $\Phi$ is a moment map for this specific choice of $\xi$. It is proper by (1.1) in \cite{Paradan99} and taming as we saw above. So Theorem \ref{thm QR=0} implies the claim.

If $\xi$ is singular, 
let $\psi\colon G/H \to \kk^*$ be given by
\[
\psi(gH) = \Ad^*(g)(\tilde \xi + \zeta)|_{\kk},
\]
for $\tilde \xi \in \kt^*_M$ such that
$(\alpha, i\tilde \xi) > 0$ for all $\alpha \in R^+_M$, and $\zeta \in \ka^*$ such that $\xi + \zeta \in \kh^*$ is regular for the roots of $(\kg^{\C}, \kh^{\C})$.
Let $\Phi^{\tau}\colon G/H \to \kk$ be as in Lemma \ref{lem taming moment}. Then $\Phi^{\tau}$ is a taming, proper moment map, and, by that Lemma and Theorem \ref{thm realise},
\[
\begin{split}
\indx_K\bigl(\Bigwedge_{J} T(G/H) \otimes L_{\lambda- \rho^M, \chi_M}, \Phi^{\tau} \bigr) &= \indx_K\bigl(\Bigwedge_{J} T(G/H) \otimes L_{\lambda- \rho^M, \chi_M}, \psi \bigr) \\
&= (-1)^{\dim(M/K_M)/2} \pi|_K.
\end{split}
\]
In the first equality we used Theorem \ref{thm htp invar}. The claim again follows from Theorem \ref{thm QR=0}.
\hfill $\square$

\begin{proof}[Proof of Corollary \ref{cor mult form K'}.]
Applying Theorem \ref{thm mult form} in this paper and Theorem 1.1 in \cite{Paradan17}, we obtain
\[
\begin{split}
\pi|_{K} &=  (-1)^{\dim(M/K'_M)/2} \bigl. \Bigl(\bigoplus_{\delta' \in \hat K'} \indx((G/H)_{(\eta'+\rho^{K'})/i}) \delta'\Bigr)\Bigr|_{K} \\
&=
 (-1)^{\dim(M/K'_M)/2} \bigoplus_{\delta \in \hat K} \indx((G/H)^{K}_{(\eta+\rho^{K})/i}) \delta.
\end{split}
\]
Here $\eta'$ is the highest weight of an irreducible representation $\delta' \in \hat K'$.
\end{proof}

\medskip
It remains to prove Proposition \ref{prop moment map} and to show how to handle the singular case, see Lemma \ref{lem taming moment}. This is done in the next section.


\section{A $\Spinc$-moment map on $G/H$}\label{sec moment GH}

\subsection{$\Spinc$-structures on linearised homogeneous spaces}\label{sec Spinc lin}

As an intermediate step in the proof of Proposition \ref{prop moment map}, we will use a $K$-equivariant partial linearisation of the space $G/H$ that was introduced in Subsection 4.2 of \cite{HSY1}. Let $\kg = \kk \oplus \ks$ be the Cartan decomposition defined by $\theta$. Write $\km = \kk_{M} \oplus \ks_M$, with $\kk_M \subset \kk$ and $\ks_M \subset \ks$. Let $H_M := H \cap M$. Then $H_M$ may be disconnected, but its Lie algebra is $\kt_M$. Let
\[
E:= K\times_{H_M} (\ks_M \oplus \kn)
\]
be the quotient of $K \times (\ks_M \oplus \kn )$ be the action by $H_M$ defined by
\[
h\cdot (k, X+Y) = (kh^{-1}, \Ad(h)(X+Y)),
\]
for $h \in H_M$, $k \in K$, $X \in \ks_M$ and $Y \in \kn$. Lemma 4.2 in \cite{HSY1} states that the map $\Psi\colon E \to G/H$ defined by
\[
\Psi([k, X+Y]) = k\exp(X)\exp(Y)H
\]
for $k \in K$, $X \in \ks_M$ and $Y \in \kn$, is a well-defined, $K$-equivariant diffeomorphism. In this sense, $E$ is a partial linearisation of $G/H$.

For every $X \in \ks_M$ and $Y \in \kn$, we have the linear isomorphism
\beq{eq iso TE}
T_{[e, X+Y]}E \xrightarrow{\cong} \kk/\kt_M \oplus \ks_M \oplus \kn = \kg/\kh
\eeq
defined by
\[
(U+\kt_M, V+W) \mapsto \ddt[\exp(tU), X+Y+t(V+W)],
\]
for $U \in \kk$, $V \in \ks_M$ and $W \in \kn$. Let $J^E$ be the $K$-invariant almost complex structure on $E$ corresponding to the complex structure $J_{\kg/\kh}$ (defined in Subsection \ref{sec tempered reps}) via the isomorphism \eqref{eq iso TE}. The almost complex structure $\Psi^*J$ on $E$ corresponding to $J$ via $\Psi$ differs from $J^E$ because it corresponds to $J_{\kg/\kh}$ via a different isomorphism \eqref{eq iso TE}. This is worked out in Subsection 4.3 of \cite{HSY1}, where it it also shown that $\Psi^*J$ and $J^E$ are $K$-equivariantly homotopic.

\begin{lemma}\label{lem gh complex}
As a complex representation of $T_M$, the space $\kg/\kh$ decomposes as
\[
\kg/\kh = \bigoplus_{\alpha \in R_G^+} {\C}_{\alpha|_{\kt_M}}.
\]
\end{lemma}
\begin{proof}
We have a complex decomposition
\[
\kg/{\kh} \cong \km/\kt_M \oplus \kn^- \oplus \kn^+ = \Bigl(\bigoplus_{\beta \in R_M^+}\km^{\C}_{\beta}\Bigr) \oplus \kn^- \oplus \kn^+.
\]
We will show that the set of weights of the action by $\kt_M$ on $\kn^- \oplus \kn^+$ is
\[
\{\alpha|_{\kt_M}; \alpha \in R_G, \alpha|_{\ka} \in \Sigma^+\} \subset i\kt_M^*.
\]
Furthermore, every weight occurs with multiplicity one. The claim then follows from \eqref{eq RG}.

 
 Let us determine the weights of the action by $\kt_M$ on $\kn^- \oplus \kn^+$. Fix $\lambda \in \Sigma^+$. Then
\[
\kg_{\lambda} = \Bigl( \bigoplus_{\alpha \in R_G, \alpha|_{\ka} = \lambda} \kg^{\C}_{\alpha}\Bigr) \cap \kg.
\]
For every $\alpha \in R_G$, we have
$
\theta \kg^{\C}_{\alpha} =\kg^{\C}_{-\bar \alpha}.
$
Since $\ad(\zeta)$ preserves $\kg^{\C}_{\alpha}$, this implies that
\[
J_{\zeta}\kg^{\C}_{\alpha} =\kg^{\C}_{-\bar \alpha}.
\]

Note that
\[
\kg_{-\lambda} =  \Bigl( \bigoplus_{\alpha \in R_G, \alpha|_{\ka} = \lambda} \kg^{\C}_{-\bar \alpha}\Bigr) \cap \kg.
\]
Here we used that $\alpha|_{\ka}$ takes values in $\R$, so $\alpha|_{\ka} = \bar \alpha|_{\ka}$.
So
\[
\kg_{\lambda} \oplus \kg_{-\lambda} = \Bigl( \bigoplus_{\alpha \in R_G, \alpha|_{\ka} = \lambda} \kg^{\C}_{\alpha}
\oplus  \kg^{\C}_{-\bar \alpha}\Bigr) \cap \kg,
\]
where every term on the right hand side is preserved by $J_{\zeta}$.

Let us determine the weight of $\kt_M$ on $\kg^{\C}_{\alpha} \oplus  \kg^{\C}_{-\bar \alpha}$, for $\alpha \in R_G$. For $Y \in \kt_M$, $X_{\alpha} \in \kg^{\C}_{\alpha}$ and $X_{-\bar \alpha} \in \kg^{\C}_{-\bar \alpha}$, we have
\[
\ad(Y)(X_{\alpha} + X_{-\bar \alpha}) = \langle \alpha, Y\rangle X_{\alpha} - \overline{\langle\alpha, Y\rangle}X_{-\bar\alpha}.
\] 
Since $\langle \alpha, Y\rangle \in i\R$, this equals
\[
\langle \alpha, Y\rangle (X_{\alpha} + X_{-\bar \alpha}).
\]
So $\kt_M$ acts on $\kg^{\C}_{\alpha} \oplus  \kg^{\C}_{-\bar \alpha}$ with weight $\alpha|_{\kt_M}$.

We therefore obtain a complex, $\kt_M$-equivariant isomorphism
\[
\begin{split}
\kn^+ \oplus \kn^- &= \bigoplus_{\lambda \in \Sigma^+} \kg_{\lambda} \oplus \kg_{-\lambda} \\
	&\cong \bigoplus_{\lambda \in \Sigma^+} \bigoplus_{\alpha \in R_G, \alpha|_{\ka} = \lambda} \C_{\alpha|_{\kt_M}} \\
	&= \bigoplus_{\alpha \in R_G, \alpha|_{\ka} \in \Sigma^+} \C_{\alpha|_{\kt_M}}.
\end{split}
\]
Here we used that the space $ (\kg^{\C}_{\alpha}
\oplus  \kg^{\C}_{-\bar \alpha}) \cap \kg$ is complex one-dimensional and preserved by $J_{\zeta}$ and $\kt_M$.
\end{proof}

Consider the spinor bundle
\[
\Bigwedge_{J^E}TE \to E
\]
of the $\Spinc$-structure defined by $J^E$. Let $L^E_{\det}$ be its determinant line bundle. Let $p\colon E \to K/H_M$ be the natural projection. Let $\chi_{\kn^- \oplus \kn^+}$ be the adjoint representation of $Z_M$ in the highest complex exterior power of $\kn^- \oplus \kn^+$.
\begin{lemma}\label{lem Ldet}
We have 
 an isomorphism of complex, $K$-equivariant line bundles
\[
L_{\det}^E \cong p^*(K\times_{H_M} \C_{2\rho^G|_{\kt_M}} \boxtimes \chi_{\kn^- \oplus \kn^+}).
\]
\end{lemma}
\begin{proof}
Lemma \ref{lem gh complex} implies that
\[
L_{\det}^E \cong p^*(K\times_{H_M} \C_{2\rho^G|_{\kt_M}} \boxtimes \chi),
\]
where $\chi$ is the representation of $Z_M$ in the highest complex exterior power of $\kg/\kh$. And $Z_M$ acts trivially on $\km/\kt_M$.
\end{proof}
\begin{remark}
Since $H_M = T_M Z_M'$ for a finite subgroup $Z_M'<Z_M$, we may replace $\chi_{\kn^- \oplus \kn^+}$ by the representation of $Z_M'$ in the highest complex exterior power of $\kn^- \oplus \kn^+$.
\end{remark}

\subsection{Line bundles}

\begin{lemma}\label{lem det GH}
The determinant line bundle of the $\Spinc$-structure on $G/H$ with spinor bundle \eqref{eq spinor GH} is
\[
G\times_H\C_{2(\rho^{G,M} + \lambda)} \boxtimes \chi_M^2 \otimes \chi_{\kn^- \oplus \kn^+}.
\]
\end{lemma}
\begin{proof}
The almost complex structures $J$ and $\Psi_* J^E$ are $K$-equivariantly homotopic (see the proof of Lemma 4.3 in \cite{HSY1}). Hence the induced $\Spinc$-structures have equivariantly isomorphic determinant line bundles.
By Lemma \ref{lem Ldet}, the determinant line bundle corresponding to $\Psi_* J^E$ is
\[
\Psi_*p^*(K\times_{H_M} \C_{2\rho^G|_{\kt}}\otimes \chi_{\kn^- \oplus \kn^+}) \to G/H. 
\]
One can check directly that this line bundle is $K$-equivariantly isomorphic to $G\times_H\C_{2\rho^G|_{\kt}} \otimes \chi_{\kn^- \oplus \kn^+}$. Therefore, the determinant line bundle of the $\Spinc$-structure in the statement of the lemma is
\[
L_{\det} = G\times_H (\C_{2\rho^G|_{\kt_M}} \otimes \C_{\lambda - \rho^M}^{2} \boxtimes \chi_M^2\otimes \chi_{\kn^- \oplus \kn^+})= G\times_H\C_{2(\rho^{G,M} + \lambda)} \boxtimes \chi_M^2\otimes \chi_{\kn^- \oplus \kn^+}.
\]
\end{proof}

\begin{lemma}\label{lem line bdles}
Let $\sigma \in i\kt^*_M$ be integral. Let $\zeta \in i\ka^*$, and set
\[
L_{\zeta} := G\times_H \C_{\sigma + \zeta}.
\]
Then there is a $K$-equivariant isomorphism of line bundles $L_{\zeta} \cong G\times_H \C_{\sigma}$.
\end{lemma}
\begin{proof}
The multiplication map defines a diffeomorphism
\[
K \times \exp(\ks_M) \times N \times A  \cong G.
\]
Define the map
\[
\Xi\colon G\times \C_{\sigma + \zeta} \to G\times \C_{\sigma } 
\]
by
\[
\Xi(khna, z) = (khna, e^{\zeta}(a)z),
\]
for $k \in K$, $h \in \exp(\ks_M)$, $n \in N$, $a \in A$, and $z \in \C_{\sigma + \zeta}$. We claim that this is map is $H$-equivariant, and that the induced map $L_{\zeta} \to G\times_H \C_{\sigma}$ is
a $K$-equivariant isomorphism of line bundles.

To show that $\Xi$ is $H$-equivariant, let $t \in T_M$ and $a_0 \in A$. Then for an element $(khna, z) \in G\times \C_{\sigma + \zeta}$  as above,
\[
(ta_0) \cdot (khna, z) = \bigl(khna(ta_0)^{-1}, e^{\sigma}(t)e^{\zeta}(a_0)z \bigr) = 
\bigl((khnt^{-1})aa_0^{-1}, e^{\sigma}(t)e^{\zeta}(a_0)z\bigr).
\]
The adjoint action by $T_M$ preserves the restricted root spaces of the system $\Sigma = \Sigma(\kg, \ka)$, because this action commutes with $\ad(\ka)$. So this action preserves $\kn$. Furthermore, since $T_M \subset K_M$, this action also preserves $\ks_M$. So if $h = \exp(X)$ and $n = \exp(Y)$, then
\[
khnt^{-1} = kt^{-1}\exp(\Ad(t)X)\exp(\Ad(t)Y)
\in  K\exp(\ks_M)N,
\]
Therefore,
\[
\begin{split}
\Xi\bigl((khnt^{-1})aa_0^{-1}, e^{\sigma}(t)e^{\zeta}(a_0)z\bigr) &= \bigl((khnt^{-1})aa_0^{-1}, e^{\zeta}(aa_0^{-1}) e^{\sigma}(t)e^{\zeta}(a_0)z\bigr) \\
&= \bigl(khna(ta_0)^{-1},   e^{\sigma}(t)e^{\zeta}(a)z\bigr) \\
&= (ta_0) \cdot \Xi(khna, z).
\end{split}
\]
%

Since $\Xi$ is $H$-equivariant, it indeed descends to a map $L_{\zeta} \to G\times_H \C_{\sigma}$. This map is immediately seen to be a $K$-equivariant isomorphism of line bundles.
\end{proof}

\begin{remark} \label{rem preq line}
Lemmas \ref{lem det GH} and \ref{lem line bdles} imply that the determinant line bundle of the $K$-equivariant $\Spinc$-structure with 
 spinor bundle \eqref{eq spinor GH} is $K$-equivariantly isomorphic to
\[
G\times_H\C_{2(\xi + \zeta)} \boxtimes \chi_M^2 \otimes \chi_{\kn^- \oplus \kn^+} \to G/H,
\]
with $\xi$ and $\zeta$ as in Proposition \ref{prop moment map}. That is to say, modulo the representation $\chi_M \otimes \chi_{\kn^- \oplus \kn^+}$ of the finite group $Z'_M$, the map $\Phi$ is the symplectic moment map for the action by $K$ on the coadjoint orbit $\Ad^*(G)(\xi+\zeta)$, while the $\Spinc$-structure with spinor bundle \eqref{eq spinor GH} is a $K$-equivariant $\Spinc$-prequantisation of this coadjoint orbit. If the infinitesimal character $\chi$ of $\pi$ is regular, then we may take $\zeta$ to be the component of $\chi$ in $\ka^*$, so that $\xi+\zeta = \chi+ \rho^{G,M}$.
\end{remark}

\subsection{Proof of Proposition \ref{prop moment map}}

We start with a general, well-known, comment about moment maps on homogeneous spaces. For now, let $G$ be any Lie group, and let $H<G$ be a possibly disconnected, closed subgroup. Let 
 $\C_{\sigma}$ be a one-dimensional unitary representation of $H$, with differential $\sigma \in i\kh^*$. Consider the line bundle
\[
L_{\sigma} := G\times_H \C_{\sigma} \to G/H.
\]
Then $\Gamma^{\infty}(L_{\sigma}) \cong (C^{\infty}(G) \otimes \C_{\sigma})^H$.
Let $V \subset \kg$ be a $H$-invariant subspace such that $\kg = \kh \oplus V$. Extend $\sigma$ linearly to $\kg$ by setting it equal to zero on $V$. 
\begin{lemma}\label{lem Spinc conn}
For $X \in \kg$ and $s \in (C^{\infty}(G) \otimes \C_{\sigma})^H$, set
\[
(\nabla_{X+\kh}s)(e) := {\calL}_{-X}(s)(e) - \langle \sigma, X\rangle s(e).
\]  
Here ${\calL}$ is the left regular representation.
This extends to a well-defined $G$-invariant connection on $L_{\sigma}$. The associated moment map $\Phi^{\sigma}\colon G/H \to \kg^*$ is given by
\[
\Phi^{\sigma}(gH) = \Ad^*(g)\sigma/2i.
\]
\end{lemma}
\begin{proof}
To see that $\nabla$ is well-defined, note that if  $s \in (C^{\infty}(G) \otimes \C_{\sigma})^H$ and $X \in \kh$, then 
$
\calR_{X}s = \langle \sigma, X\rangle s,
$
with $\calR$ the right regular representation.
So, at $e$,
\[
{\calL}_{-X}(s)(e) - \langle \sigma, X\rangle s(e) = \calR_X(s)(e) - \langle \sigma, X\rangle s(e) = 0.
\] 

The moment map $\Phi^{\sigma}$ satisfies
\[
2i \langle \Phi^{\sigma},X\rangle s(e)= ({\calL}_X - \nabla_{X^{G/H}})s(e) 
\]
Now note that with respect to the identification $T_{eH}G/H = \kg/\kh$,
\[
X^{G/H}_{eH} = \ddt \exp(-tX)H = -X+\kh.
\]
Hence
\[
 ({\calL}_X - \nabla_{X^{G/H}})s(e) =  ({\calL}_X - \nabla_{-X+\kh})s(e) = \langle\sigma, X \rangle s(e).
\]
So $\Phi^{\sigma}(e) = \sigma/2i$, and the claim about $\Phi^{\sigma}$ follows by $G$-equivariance.
\end{proof}

Importantly, even if $H$ is disconnected, so the representation $\C_{\sigma}$ of $H$ is not determined by $\sigma$, Lemma \ref{lem Spinc conn} still gives us a connection with the desired moment map. This means we can apply it to the representation $\C_{2\rho^G - 2\rho^M +2\lambda + i\zeta}  \boxtimes \chi_M^2 \otimes \chi_{\kn^- \otimes \kn^+}$ of the Cartan subgroup $H$.


\medskip \noindent
\emph{Proof of Proposition \ref{prop moment map}.}
By Lemmas \ref{lem det GH} and \ref{lem line bdles}, we have $K$-equivariant isomorphisms of line bundles
\[
L_{\det} \cong G\times_H \C_{2(\rho^{G,M} +\lambda)} \boxtimes \chi_M^2\otimes \chi_{\kn^- \otimes \kn^+}\cong G\times_H \C_{2(\rho^{G,M} +\lambda+i\zeta)} \boxtimes \chi_M^2\otimes \chi_{\kn^- \otimes \kn^+}.
\]
Let $\nabla$ be the connection of Lemma \ref{lem Spinc conn} on the line bundle on the right hand side;  we use the same notation for the connection on $L_{\det}$ corresponding to $\nabla$ via the above isomorphism. The moment map for the action by $K$ on $G/H$ associated to $\nabla$ is the map $\Phi^{2(\rho^{G,M} +\lambda+i\zeta)}$ in Lemma \ref{lem Spinc conn}, composed with restriction to $\kk$. This is precisely the map $\Phi$ in Proposition \ref{prop moment map}.
\hfill $\square$

\begin{remark}
In the proof of Proposition \ref{prop moment map}, we used Lemma \ref{lem line bdles} to replace $2(\rho^{G,M} +\lambda)$ by $2(\rho^{G,M} +\lambda+i\zeta)$. The reason for introducing the extra term  in Proposition \ref{prop moment map} is that the moment map $\Phi$ is taming 
if $\rho^{G,M} +\lambda+i\zeta$ is regular.
\end{remark}

\subsection{The singular case} \label{sec sing}


If $\xi + \zeta$ is singular, then the moment map of Proposition \ref{prop moment map} is not necessarily proper or taming. But then we can still find a proper, taming $\Spinc$-moment map such that the associated index equals $(-1)^{\dim(M/K_M)}\pi|_K$.


Consider a general setting, where $M$ is a complete Riemannian manifold with an action by a compact Lie group $K$, and $\Phi\colon M\to \kk$ is the moment map for a connection $\nabla$ on a line bundle (defined as in \eqref{eq def moment}), and $\psi\colon M \to \kk$ is a taming map. For $\tau \in C^{\infty}(M)^K$, define the connection
\[
\nabla^{\tau} := \nabla + 2i{\tau}(v^{\psi}, \relbar).
\]
Let $\Phi^{\tau}$ be the associated moment map.
\begin{lemma}\label{lem taming moment}
For ${\tau}$ large enough, the map $\Phi^{\tau}$ is proper, taming, and homotopic to $\psi$ as taming maps.
\end{lemma}
\begin{proof}
Let $\{X_1, \ldots, X_n\}$ be an orthonormal basis of $\kk$. Then
\[
v^{\Phi^{\tau}} = v^{\Phi} + {\tau} w^{\psi},
\]
where
\[w^{\psi} :=
\sum_{j=1}^n (v^{\psi}, X_j^M)X_j^M.
\]
Let $m \in M$, and suppose $v^{\psi}(m) \not=0$. The definition of $w^{\psi}$ is independent of the basis of $\kk$, so we may suppose that $X_1, \ldots, X_d \in \kk_m$ and $X_{d+1}, \ldots, X_n \in \kk_m^{\perp}$. 
Since $\{X_{d+1}^M(m), \ldots, X_n^M(m)\}$ is a basis of the subspace $T_m(K\cdot m) \subset T_mM$ containing $v^{\psi}(m)$, we have that
\[
w^{\psi}(m) = 
\Bigl(\sum_{j=1}^n (v^{\psi}, X_j^M)X_j^M\Bigr)(m) =  \Bigl(\sum_{j=d+1}^n (v^{\psi}, X_j^M)X_j^M\Bigr)(m) \not= 0.
\] 
So $w^{\psi}$ vanishes exactly at the points where $v^{\psi}$ vanishes.

Now note that
\[
\|v^{\Phi^{\tau}}\| \geq {\tau}\|w^{\psi}\| - \|v^{\Phi}\|.
\]
Let $U$ be a relatively compact, $K$-invariant neighbourhood of the vanishing set of $v^{\psi}$. Choose ${\tau}$ so that, outside $U$, ${\tau}\|w^{\psi}\| > \|v^{\Phi}\|$. Then $v^{\Phi^{\tau}}$ does not vanish outside $U$.

To show that $\Phi^{\tau}$ is homotopic to $\psi$, first note that by the previous arguments, the vector field
\[
tv^{\Phi} + \tau w^{\psi}
\]
is nonzero outside $U$, for all $t \in [0,1]$. Hence $\Phi^{\tau}$ is homotopic to the taming map $\Phi^{\tau} - \Phi$. And
\[
(v^{\Phi^{\tau} - \Phi}, v^{\psi}) = (\tau w^{\psi}, v^{\psi}) = {\tau} \sum_{j=1}^n (v^{\psi}, X_j^M)^2.
\]
If ${\tau}\geq 0$, then this is nonnegative. This implies that $\Phi^{\tau} - \Phi$ is homotopic to $\psi$ (this is elementary, see for example Corollary 3.5 in \cite{HSY1}).

Finally, by adding a function $\theta$ as in Subsection 5.1 of \cite{HS16} to $\tau$, we can ensure that the resulting moment map is proper, as well as taming, and homotopic to $\psi$.
\end{proof}


\begin{remark}
In 
Proposition 5.1 in \cite{HS16}, it was shown how to replace a taming moment map by a proper, taming moment map, without changing the corresponding indices. (This was the last step in the proof of Lemma \ref{lem taming moment}.) In Lemma \ref{lem taming moment}, we show how to replace any taming map by a proper, taming moment map without changing the index. The additional step here is to replace any taming map by a taming moment map that is homotopic to it.
\end{remark}

\section{Multiplicity-free restrictions} \label{sec mult free}


Throughout this section, $K<G$ is a compact subgroup satisfying the condition of Corollary \ref{cor mult form K'}. That is, the map $\Phi_K$ in \eqref{eq Phi K} is proper. In particular, what follows is true if $K$ is a maximal compact subgroup, as we will assume from Subsection \ref{sec SL2} onwards. We will omit the subscript $K$ from $\Phi_K$, and write $\Phi:= \Phi_K$ from now on. We will also write $(G/H)_{\xi} := (G/H)^K_{\xi}$ for $\xi \in \kt^*$.


Recall that if $\xi \in \kt_M^*$, defined in \eqref{eq def xi}, is regular for the roots of $(\km^{\C}, \kt_M^{\C})$, then $\Phi$ is simply the projection of the coadjoint orbit $\Ad^*(G)(\xi + \zeta)$ onto $\kk^*$:
\beq{eq Phi reg}
\Phi(gH) = (\Ad^*(g)(\xi + \zeta))|_{\kk}
\eeq
If $\xi$ is singular, then $\Phi$ is as in \eqref{eq mu tau GH}, with $\tau$ as in
Lemma \ref{lem taming moment}.

\subsection{Reduced spaces that are points}


In the setting of Corollary \ref{cor mult form K'}, we obtain multiplicities equal to $0$ or $1$ if 
 the reduced space $(G/H)_{(\eta+\rho^K)/i}$ is a single point. Indeed,  
 the orbifold index on $(G/H)_{(\eta+\rho^K)/i}$ then lies in $\{-1,0,1\}$. It takes only these values, because, up to a sign, it is the dimension of the trivial part of a one-dimensional representation of a finite group.
We can make this more explicit using the expression  (5.34) in \cite{Paradan14} for indices on reduced spaces that are points.

%

Let $C \subset \kt^*$ be the open positive Weyl chamber. Set $Y := \Phi^{-1}(C)$. 
Set $\Phi_Y := \Phi|_Y - \rho^K$.
Let $\delta \in \hat K$ have highest weight $\eta$.
 Then
\[
(G/H)_{(\eta + \rho^K)/i} = Y_{\eta/i} := \Phi_Y^{-1}(\eta/i)/T.
\]
Let  $\kt_Y \subset \kt$ be the generic stabiliser of the infinitesimal action by $\kt$ on $Y$.  Let $I(Y) \subset \kt^*$ be the affine space parallel to the annihilator of $\kt_Y$, containing the image of $\Phi_Y$. Let $T_Y < T$ be the subtorus with Lie algebra $\kt_Y$ (note that this subgroup is connected). Fix $g_0H \in \Phi_Y^{-1}(\eta/i)$, and let $\Gamma < T/T_Y$ be its stabiliser. This is a finite group.
\begin{corollary}\label{cor one 2}
\begin{enumerate}
\item[(a)]
Suppose that $\eta/i$ is a regular value of $\Phi_Y\colon Y \to I(Y)$, and that 
%
$\Phi^{-1}(\eta/i+ \rho^K)/K$ is a point. 
%
%
Then
\[
[\pi|_K:\delta] = \left\{
\begin{array}{ll}
1 & \text{if $\Gamma$ acts trivially on $\C_{\lambda  - \eta - \rho^M} \boxtimes \chi_M$};\\
0 & otherwise.
\end{array}
\right.
\]
\item[(b)]
If $\eta/i$  is not necessarily a regular value of $\Phi_Y$, but 
%
$\Phi^{-1}(\eta/i + \rho^K + \varepsilon)/K$ is a point
for all $\varepsilon \in I(Y)$ close enough to $0$, then we still have
\[
[\pi|_K:\delta] \in \{0,1\}.
\]
\end{enumerate}
\end{corollary}
\begin{proof}
First of all, note that for any $\sigma \in \kt^*$, by construction $\Phi^{-1}(\sigma+\rho^K)$ is a single $K$-orbit if and only if $\Phi_Y^{-1}(\xi)$ is a single $T$-orbit.

We have
\[
T(G/H)|_{Y} = TY \oplus (Y \times \kk/\kt).
\]
By the two-out-of-three lemma, we have a spinor bundle
 $\calS_{Y,\eta} \to Y$ such that
\beq{eq def SY}
\bigl(\Bigwedge_J T(G/H) \otimes L_{\lambda - \rho^M, \chi_M}\otimes \C_{-\eta}\bigr)|_Y = \calS_{Y,\eta} \otimes \Bigwedge_{\C} \kk/\kt.
\eeq
Here the complex structure on $\kk/\kt$ is the one defined by the positive compact roots. Let $V_{\eta}$ be the one-dimensional representation of $\Gamma$ such that, as representations of $\Gamma$,
\beq{eq def V}
(\calS_{Y, \eta})_{g_0H} = \Bigwedge_{\C} T_{g_0H}Y \otimes V_{\eta},
\eeq
for some $\Gamma$-invariant complex structure on $T_{g_0H}Y$. This $V_{\eta}$ exists since $(\calS_{Y, \eta})_{g_0H} $ and $\Bigwedge_{\C} T_{g_0H}Y$ are irreducible, $\Gamma$-equivariant modules over the Clifford algebra of $T_{g_0H}Y$; see also (5.33) in  \cite{Paradan14}.
Then (5.34) in \cite{Paradan14} states that
\beq{eq indx red pt}
\indx((G/H)_{(\eta+\rho^K)/i}) = \indx (Y_{\eta}) =  \dim V_{\eta}^{\Gamma}.
\eeq

Now by \eqref{eq def SY} and \eqref{eq def V},
\[
\Bigwedge_{J_{g_0H}} (T_{g_0H}G/H)\otimes \C_{\lambda - \eta -\rho^M} \otimes \chi_M  =
\Bigwedge_{\C} (T_{g_0H}G/H)\otimes V_{\eta}.
\]
Here on the right hand side, the complex structure on $T_{g_0H}G/H$ is defined by the complex structures on $ T_{g_0H}Y$ and $\kk/\kt$ via the isomorphism $T_{g_0H}G/H \cong T_{g_0H}Y \oplus \kk/\kt$. This may be a different complex structure from $J_{g_0H}$. We conclude that $V_{\eta}$ equals $\C_{\lambda- \eta - \rho^M } \otimes \chi_M$ or its dual. So the claim follows from \eqref{eq indx red pt} and Corollary \ref{cor mult form K'}.
%
%
%
\end{proof}
\begin{remark}
We have implicitly used that $\Phi_Y^{-1}(\eta)$ is connected, because it is a single $T$-orbit.
\end{remark}

\begin{example}
If $G = \SL(2,\C)$, then one can check that all reduced spaces are points. This is compatible with the fact that the multiplicities of the $K$-types of the principal series of $\SL(2,\C)$ are $1$. 
\end{example}
We work out the example $G = \SL(2,\R)$ in detail in Subsection \ref{sec SL2}, and discuss the groups
$\SU(p,1)$, $\SO_0(p,1)$ and $\SO_0(2,2)$ in Subsection \ref{sec ex mult free}. 


Corollary \ref{cor one 2} is closely related to the \emph{Corwin--Greenleaf multiplicity function} \cite{CG88}. This is the function $n\colon \kg^*/G \times \kk^*/K \to \Z_{\geq 0}$ given by
\[
n(\cO^G, \cO^K) = \# (\cO^G \cap p^{-1}(\cO^K))/K,
\]
where $p\colon \kg^*\to \kk^*$ is the restriction map.
If $\pi$ has regular infinitesimal character, then reduced spaces of the action by $K$ on $G/H$ are of the form 
 $(\cO^G \cap p^{-1}(\cO^K))/K$ above. Hence these space are points precisely if the Corwin--Greenleaf function of the corresponding orbits equals $1$, and empty if that function gives $0$. Together with a result by Nasrin, proving a special case of a conjecture by Kobayashi, this gives the following result.
 \begin{corollary} \label{cor CG} Suppose $K<G$ is maximal compact.
 Suppose that $\xi$ is regular, and let $\zeta$ be as in \eqref{eq Phi reg}. Suppose that
 \beq{eq cond Nasrin}
 \Ad^*(G)(\xi + \zeta) \cap ([\kk,\kk] + \kp)^{\perp} \not = \emptyset.
 \eeq
 Let $\delta \in \hat K$ have highest weight $\eta$.
If $\eta/i$ is a regular value of $\Phi_Y\colon Y \to I(Y)$, then
\[
[\pi|_K : \delta] \in \{0,1\}.
\]
The condition in Corollary \ref{cor one 2}(a) determines when this multiplicity is $1$.
 \end{corollary}
 \begin{proof}
 Nasrin proved that the condition \eqref{eq cond Nasrin} implies that $n( \Ad^*(G)(\xi + \zeta), \cO^K) \leq 1$ for all coadjoint orbits $\cO^{K} \in \kk^*/K$. See Theorem 1.3 in \cite{Nasrin10}. By Corollaries \ref{cor zero} and \ref{cor one 2}, this implies the claim.
 \end{proof}
Kobayashi conjectured that \eqref{eq cond Nasrin} implies that $n( \Ad^*(G)(\xi + \zeta), \cO^H) \leq 1$ for all coadjoint orbits $\cO^H$ of a subgroup $H<G$ such that $(G,H)$ is a symmetric pair. Nasrin's result used in the above proof shows that this conjecture is true for $H=K$. Note that The condition \eqref{eq cond Nasrin} can only hold if $G$ is of Hermitian type; i.e. $[\kk, \kk] \not= \kk$. A restatement of \eqref{eq cond Nasrin} is that $\Ad^*(G)(\xi + \zeta)$ is a coadjoint orbit through a central element of $\kk^* \cong \kk$.

Finally, we obtain a criterion for multiplicity-free restrictions of general \emph{admissible} representations. Let $\pi$ be an irreducible admissible representation of $G$. By the Langlands classification, $\pi$ is a quotient of an induced representation as on the right hand side of \eqref{eq pi ind}, where now $\nu \in (\ka^{\C})^*$ may be non-imaginary.
Let $\Phi\colon G/H \to \kk^*$ be the corresponding moment map as in Proposition \ref{prop moment map}. 
\begin{corollary} \label{cor admissible}
Let $\delta \in \hat K$, with highest weight $\eta$. 
Suppose that  $\Phi^{-1}(\eta + \rho^K)/K$ is a point if $\eta/i$ is a regular value of $\Phi_Y$, or $\Phi^{-1}(\eta + \rho^K + \varepsilon)/K$ is a point for all $\eta$ small enough if $\eta/i$ is a singular value of $\Phi_Y$. Then $[\pi|_K, \delta] \in \{0,1\}$.

In particular, if all reduced spaces for $\Phi$ are points, then $\pi$ restricts multiplicity-freely to $K$.
\end{corollary}
\begin{proof}
Corollary \ref{cor mult form K'}, and hence Corollary \ref{cor one 2}, apply to any standard representation $\pi$; see Remark \ref{rem standard reps}. So under the conditions stated, $\pi|_K$
 is a quotient of a multiplicity-free representation, and hence multiplicity-free itself.
 \end{proof}

We end this subsection with a conjecture that is a partial converse to Corollary \ref{cor one 2}.
\begin{conjecture}\label{conj mult free}
Let $H<G$ be a $\theta$-stable Cartan subgroup. Let $P = MAN <G$ be a cuspidal parabolic subgroup corresponding to $H$ (so that $A$ is the noncompact part of $H$). Then all tempered representations $\pi$ induced from $P$ restrict multiplicity-freely to $K$ if and only if all reduced spaces for all maps $\Phi\colon G/H \to \kk^*$ corresponding to such representations are points.
\end{conjecture}
The `if' part of this conjecture follows from Corollary \ref{cor one 2}. Evidence for the `only if' part is the following. Let $\delta \in \hat K$ have highest weight $\eta$. Let $H$, $\pi$ and $\Phi$ be as in the conjecture.
If the reduced space $(G/H)_{\eta + \rho^K}$ is smooth, the Atiyah--Singer index theorem and Theorem \ref{thm mult form} imply that
\[
[\pi|_K: \delta] = (-1)^{\dim(M/K_M)/2} \int_{(G/H)_{\eta + \rho^K}} e^{\frac{1}{2}c_1(L_{\det}^{{\eta + \rho^K}})} \hat A((G/H)_{\eta + \rho^K}).
\]
Here $L_{\det}^{\eta + \rho^K} \to (G/H)_{\eta + \rho^K}$ is induced by the determinant line bundle on $G/H$ from Lemma \ref{lem det GH}. If $(G/H)_{\eta + \rho^K}$ is not a point, then the right hand side depends on $c_1(L_{\det}^{{\eta + \rho^K}})$. Then one expects that number to vary with $\pi$ and $\delta$, and hence not to equal $1$ for all $\pi$ and $\delta$.

\subsection{Example: $G = \SL(2,\R)$} \label{sec SL2}

If $G = \SL(2,\R)$ and $K = \SO(2)$, then Theorem \ref{thm mult form} implies the usual multiplicity formulas for the $K$-types of tempered representations of $\SL(2,\R)$. This example illustrates the essential point that indices on reduced spaces that are points may be zero (as in Corollary \ref{cor one 2}), because these indices are orbifold indices.

\subsubsection{The discrete series}

Consider the holomorphic discrete series representation $D_n^+$ of $G = \SL(2, \R)$, for $n \in \{1, 2, 3, \ldots\}$. Then $H = T = \SO(2)$,  $M = G$, and $\lambda = n\alpha/2$, where $\alpha \in i\kt^*$ is the root mapping $\mattwo{0}{-1}{1}{0}$ to $2i$. So $\rho^G = \rho^M$, and $\xi = n\alpha/2i$. This element is regular, so $\Phi$ is the projection of $G/T\cong G\cdot \xi$ onto $\kk^*$.

Let $\delta_l = \C_l$ be the irreducible representation of $K = \SO(2)$ with weight $l \in \Z$; i.e.\ $\C_l = \C_{l\alpha/2}$.
 If $l\leq n$, then by Corollary \ref{cor zero},
\[
[D^+_n:\delta_l] = 0.
\]
If $l > n$, then $l\alpha/2i$ is a regular value of $\Phi$, and $\Phi^{-1}(l\alpha/2i)$ is a circle, acted on by $T = \SO(2)$ by rotations with weight $2$. Now $\kt_Y = \{0\}$, so $T_Y = \{I\}$ and $\Gamma = \{\pm I \}$ in Corollary \ref{cor one 2}. 
Since $Z_M \subset T_M=T$, we have $\C_{\lambda - \rho^M} \otimes \chi_M = \C_{\lambda-\rho^M} = \C_{(n-1)\alpha/2}$. Hence $\Gamma$ acts trivially on
\[
\C_{\lambda - \eta - \rho^M} \otimes \chi_M = \C_{n-l-1}
\]
precisely if $n-l$ is odd. We conclude that
\[
[D^+_n:\delta_l] = \left\{ 
\begin{array}{ll}
1 & \text{if $l = n+s$ for a positive odd integer $s$;}\\
0 & \text{otherwise}.
\end{array}
\right.
\]

In the same way, we find that for the antiholomorphic discrete series representation $D_n^-$,
\[
[D^-_n:\delta_l] = \left\{ 
\begin{array}{ll}
1 & \text{if $l = -n-s$ for a positive odd integer $s$;}\\
0 & \text{otherwise}.
\end{array}
\right.
\]

See Example 2.21 in \cite{Paradan11} for a 
 symplectic version of the computation of indices on reduced spaces in this example. 

\subsubsection{Limits of discrete series}

Consider the limit of discrete series representation $D_0^+$. Then, as in the discrete series case, $H = T = \SO(2)$ and  $M = G$. But now $\lambda = 0$, which is singular. So we have to use the taming moment map from Lemma \ref{lem taming moment}. Taking $\psi(gT) = (\Ad^*(g)\alpha/2i)|_{\kk}$, we have for all $\tau \in C^{\infty}(G/T)^K$,
\[
\Phi_X = \Phi^{\tau}_X = \tau \cdot (v^{\psi}, X^{G/T}).
\]
Let $\varphi\colon G/T \to [1,\infty[$ be the function such that for all $ g\in G$,
\[
\psi(gT) = \varphi(gT)\alpha/2i.
\]
Then
\[
\Phi^{\tau} = \tau \varphi \| (\alpha/2i)^{G/T}\|^2 \alpha/2i. 
\]
The factor $ \| (\alpha/2i)^{G/T}\|^2$ only vanishes at the point $eT$. So we can choose $\tau$ so that $\Phi^{\tau} = f\alpha/2i$ for a surjective, proper, $K$-invariant map $f \colon G/T \to [0,\infty[$ whose level sets are circles. (In fact we may take $\tau \equiv 1$.)
Then $\Phi^{\tau}$ is $K$-invariant, proper, taming, homotopic to $\psi$, and surjective onto the closed Weyl chamber containing $\alpha$. For all integers $l\geq 1$, $l\alpha/2i$ is a regular value of $\Phi^{\tau}$, and $(\Phi^{\tau})^{-1}(l\alpha/2i)$ is a circle, 
 acted on by $T = \SO(2)$ with rotations with weight $2$. So in the same way as for the discrete series, we find that
\[
[D^+_0:\delta_l] = \left\{ 
\begin{array}{ll}
1 & \text{if $l$ is a positive odd integer;}\\
0 & \text{otherwise}.
\end{array}
\right.
\]
And analogously,
\[
[D^-_0:\delta_l] = \left\{ 
\begin{array}{ll}
1 & \text{if $l$ is a  negative odd integer;}\\
0 & \text{otherwise}.
\end{array}
\right.
\]

\subsubsection{The principal series}

Consider the spherical principal series representation $P^+_{i\nu}$, for $\nu\geq0$. We now have 
\[
H =  \Bigl\{\mattwo{x}{0}{0}{x^{-1}}; x\not=0 \Bigr\},
\]
$M = \{\pm I\}$, $\lambda = 0$, and $\chi_M = \chi_+$, the trivial representation of $M$. 
Now $\kt_M = 0$, so 
 $\xi =0$.
For any nonzero  $\zeta   \in \ka$, the element $\xi + \zeta = \zeta$ is regular.
So $\Phi\colon G/H \to \kk^*$ is the projection map of the hyperbolic coadjoint orbit $G/H \cong G\cdot \zeta$ onto $\kk^*$. Therefore, for all $l \in \Z$, $\Phi^{-1}(l\alpha/2i)$ is a circle, on which $T = \SO(2)$ acts by rotations with weight $2$. Also, $l\alpha/2i$ is a regular value of $\Phi$.

In Corollary \ref{cor one 2}, we have
\[
\C_{\lambda - \eta -\rho^M} \otimes \chi_M |_{\Gamma}= \C_{-l} \otimes \chi_+|_{\Gamma} = \C_{-l}|_{\Gamma}
\] 
The group $\Gamma = \{\pm I\}$ acts trivially on this space precisely if $l$ is even. Hence
\[
[P^+_{i\nu}:\delta_l] = \left\{ 
\begin{array}{ll}
1 & \text{if $l$ is even;}\\
0 & \text{otherwise}.
\end{array}
\right.
\]

For a nonspherical principal series representation $P^-_{i\nu}$ (with $\nu>0$), we have $\chi_M  = \chi_-$, the nontrivial representation of $Z_M = M$. Hence
\[
\C_{\lambda - \eta- \rho^M} \otimes \chi_M |_{\Gamma}= \C_{-l} \otimes \chi_- |_{\Gamma} = \C_{-l+1}|_{\Gamma}
\]
Now $\Gamma = \{\pm I\}$ acts trivially on this space precisely if $l$ is odd. Hence
\[
[P^-_{i\nu}:\delta_l] = \left\{ 
\begin{array}{ll}
1 & \text{if $l$ is odd;}\\
0 & \text{otherwise}.
\end{array}
\right.
\]

\subsection{Multiplicity-freeness via dimension counts} \label{sec dim count}

Corollary \ref{cor one 2} implies a dimension-counting criterion  for the restriction of any admissible representation to $K$ to be multiplicity-free. Let $\pi$ be an admissible representation. By the Langlands classification of admissible representations, and the fact that any tempered representation is a subrepresentation of a representation induced from a discrete series representation, $\pi$ is a subrepresentation of a quotient of a representation of the form
\[
\tilde \pi := \Ind_{MAN}^G(\pi^M_{\lambda, \chi_M} \otimes e^{\nu} \otimes 1_N),
\]
for a cuspidal parabolic $MAN < G$, where $\pi^M_{\lambda, \chi_M}$ belongs to the discrete series of $M$, and $\nu \in (\ka^{\C})^*$ may be non-imaginary. Let $\Phi\colon G/H \to \kk^*$ be the moment map from Proposition \ref{prop moment map} for this situation.
We write $\dim(\im(\Phi))$ for the dimension of the relative interior of $\im(\Phi)$.
\begin{corollary} \label{cor mult free}
If
\beq{eq dim im Phi}
\dim(\im(\Phi))=
\dim(G)-\rank(G)-\dim(T), 
\eeq
then $[\pi|_K:\delta] \in \{0,1\}$ for all $\delta \in \hat K$.

In particular, if $\im(\Phi)$ has nonempty interior in $\kk^*$, and 
\[
\dim(G) \leq \rank(G)+\dim(T)+\dim(K),
\]
then $[\pi|_K:\delta] \in \{0,1\}$ for all $\delta \in \hat K$.
%
\end{corollary}
\begin{proof}
For a map $\Phi$ as in the corollary, 
the condition \eqref{eq dim im Phi} implies that the reduced space $\Phi^{-1}(\sigma)/T$ is zero-dimensional for every $\sigma$ in the relative interior of $\im(\Phi)$. Since $\xi \in \kt_M^*$ is regular,  $\Phi$ is a moment map in the symplectic sense, so $\Phi^{-1}(\sigma)/T$ is connected for such $\sigma$, hence a point. So by Corollary \ref{cor one 2}(b), which applies to representations like $\tilde \pi$, that representation restricts multiplicity-freely to $K$. Hence so does $\pi$. 
\end{proof}

In Subsection \ref{sec ex mult free}, we show that admissible representations of $\SU(p,1)$, $\SO_0(p,1)$ and $\SO_0(2,2)$ with regular infinitesimal characters have multiplicity-free restrictions to maximal compact subgroups. This is based on Corollary \ref{cor mult free} and techniques for computing the dimension of the image of $\Phi$ developed in Subsection \ref{sec dim im Phi}. 

From now on, suppose that $K<G$ is a maximal compact subgroup.
The condition 
\[
\dim(G) \leq \rank(G)+\dim(T)+\dim(K),
\]
in Corollary \ref{cor mult free} holds for the following classical semisimple groups:
\begin{itemize}
\item $\SL(2,\C)$;
\item $\SO(n,\C)$ for $n \leq 4$;
\item $\SL(2, \HH)$; 
\item $\SL(2,\R)$;
\item $\SO^*(4)$;
\item $\SU(p,1)$ for all $p$; 
\item $\SO_0(p,1)$ for all $p$,  and $\SO_0(2,2)$.
\end{itemize}
So for these groups, any admissible representation for which $\im(\Phi)$ has nonempty interior in $\kk^*$ has multiplicity-free restriction to a maximal compact subgroup.
To determine the dimension of the image of $\Phi$, we use the equality
\[
\dim(\im(\Phi)) = \dim(K/T) + \dim(\im(\Phi) \cap \kt).
\]

\subsection{Computing the dimension of $\im(\Phi)$}\label{sec dim im Phi}

The following proposition is a tool to compute $\dim(\im(\Phi) \cap \kt)$.

Let $\kh_c\subset \kg$ be a maximally compact, $\theta$-stable Cartan subalgebra. Let $R^+_n\subset R(\kg^{\C}, \kh^{\C}_c)$ be a choice of positive, imaginary, noncompact roots. For every $\alpha \in R^+_n$, let $E_{\alpha} \in \kg^{\C}_{\alpha}$ be any nonzero vector. Let $\bar E_{\alpha}$ be its complex conjugate with respect to the real form $\kg$, and set $H_{\alpha} := [E_{\alpha}, \bar E_{\alpha}]$. 
\begin{proposition}\label{prop dim im mu}
Let $\kh \subset \kg$ be
any $\theta$-stable Cartan subalgebra.
Suppose that $\Phi\colon G/H \to \kk$ is given by
\[
\Phi\colon G/H \cong \Ad(G)(\xi+\zeta) \hookrightarrow \kg \to \kk,
\]
for $\xi \in \kt \cap \kh$ and $\zeta \in \ka = \kh \cap \kp$ such that $\xi + \zeta \in \kh$ is regular. Then
$\im(\Phi) \cap \kt$ contains the convex hull of the set
\[
\bigcup_{\alpha \in R^+_n} \xi + I_{\alpha}iH_{\alpha}, 
\]
where for all $\alpha \in R^+_n$, the set $I_{\alpha}$ equals either $\R$, $[0,\infty)$ or $(-\infty, 0]$.
\end{proposition}

We will use Lemmas \ref{lem im mu 3} and \ref{lem im mu 1} below to prove Proposition \ref{prop dim im mu}.
\begin{lemma}\label{lem im mu 3}
Consider the map
\[
\Phi\colon G/H \cong \Ad(G)(\xi+\zeta) \hookrightarrow \kg \to \kk,
\]
for $\xi \in \kt \cap \kh$ and $\zeta \in \ka = \kh \cap \kp$ such that $\xi + \zeta \in \kh$ is regular.
Suppose 
 that there is a set of roots $S \subset R(\kg^{\C}, \kh^{\C})$, and for every $\alpha \in S$, there are $X_{\pm \alpha} \in \kg^{\C}_{\pm \alpha}$ such that
\begin{itemize}
\item $X_{\alpha} + X_{-\alpha} \in \kg$;
\item $X_{\alpha}- X_{-\alpha} \in \kt$;
\item $\eta_{\alpha} := [X_{\alpha}, X_{-\alpha}] \in \ka$.
\item $\langle \alpha, \eta_{\alpha}\rangle >0$.
\end{itemize}
Then $\im(\Phi) \cap \kt$ contains
\[
\xi + \Span_{\R}\{(X_{\alpha} - X_{-\alpha}); \alpha \in S\}.
\]
\end{lemma}
\begin{proof}
Fix $\alpha \in R(\kg^{\C}, \kh^{\C})$ and $X_{\pm \alpha} \in \kg^{\C}_{\pm\alpha}$. Write $\eta_{\alpha} := [X_{\alpha}, X_{-\alpha}]$.
One proves by induction that for every positive integer $j$,
\[
\begin{split}
\ad(X_{\alpha}+ X_{-\alpha})^{2j}(\xi+\zeta) &= 2^j\langle \alpha, \xi+\zeta \rangle \langle\alpha, \eta_{\alpha} \rangle^{j-1}\eta_{\alpha};\\
\ad(X_{\alpha}+ X_{-\alpha})^{2j+1}(\xi+\zeta) &= -2^j\langle \alpha, \xi +\zeta \rangle \langle\alpha, \eta_{\alpha} \rangle^{j}(X_{\alpha}-X_{-\alpha}).
\end{split}
\]
Suppose $\langle \alpha, \eta_{\alpha}\rangle > 0$. 
Then the above equalities imply that for all $t \in \R$,
\begin{multline} \label{eq Ad X alpha 2}
\Ad(\exp(t(X_{\alpha}+X_{-\alpha})))(\xi+\zeta) 
	=\\
	 \xi + \zeta +  \frac{\langle \alpha, \xi +\zeta \rangle }{\langle \alpha, \eta_{\alpha} \rangle } \eta_{\alpha}
		 \sum_{j=1}^{\infty} \frac{1}{(2j)!} t^{2j} 2^j \langle \alpha, \eta_{\alpha} \rangle^{j} \\
		  -\langle \alpha, \xi+\zeta \rangle (X_{\alpha} - X_{-\alpha})\sum_{j=0}^{\infty} \frac{1}{(2j+1)!} t^{2j+1}2^j \langle \alpha, \eta_{\alpha}\rangle\\
	= \xi + \zeta + \langle \alpha, \xi +\zeta \rangle\Bigl(
	  \frac{\cosh(t\sqrt{2\langle \alpha, \eta_{\alpha}\rangle} ) -1}{\langle \alpha, \eta_{\alpha} \rangle } \eta_{\alpha}
-\frac{\sinh(t\sqrt{2\langle \alpha, \eta_{\alpha}\rangle} )}{\sqrt{2\langle \alpha, \eta_{\alpha} \rangle}} (X_{\alpha} - X_{-\alpha}) 
	\Bigr).
\end{multline}

%

Suppose $\alpha \in S$, and let $X_{\pm \alpha}$ be as in the lemma. Then, using \eqref{eq Ad X alpha 2} and the fact that both sides of this equality lie in $\kg$ (so the component of the right hand side in $i\kg$ is zero), we find that
\[
\Phi(\exp(t(X_{\alpha}+X_{-\alpha}))H) =  \xi -\frac{\sinh(t\sqrt{2\langle \alpha, \eta_{\alpha}\rangle} )}{\sqrt{2\langle \alpha, \eta_{\alpha} \rangle}} (X_{\alpha} - X_{-\alpha}) \quad \in \kt. 
\]
So 
\[
\xi + \R(X_{\alpha} - X_{-\alpha}) \in \im(\mu) \cap \kt.
\]
And since $\Phi$ is a moment map in the symplectic sense, its image intersected with $\kt$ is convex.
%
%
%
%
%
%
\end{proof}

\begin{example}
If $G = \SL(2,\R)$, $\xi = 0$, $\zeta = \begin{pmatrix} 1 & 0 \\ 0 & -1\end{pmatrix}$, $\kh = \ka = \R\zeta$, $\langle \alpha, \zeta \rangle = 2$, $S = \{\alpha\}$, and
\[
X_{\alpha} =  \begin{pmatrix} 0 & 1 \\ 0 & 0\end{pmatrix}; \qquad X_{-\alpha} =  \begin{pmatrix} 0 & 0 \\ 1 & 0\end{pmatrix},
\]
then $X_{\alpha} - X_{-\alpha} \in \kt$, $\eta_{\alpha} = \zeta \in \ka$, and Lemma \ref{lem im mu 3} states that $\im(\Phi)$ contains the line $\R\zeta$, and is therefore surjective.
\end{example}

\begin{lemma}\label{lem im mu 1}
Consider the map
\[
\Phi\colon G/H \cong \Ad(G)(\xi+\zeta) \hookrightarrow \kg \to \kk,
\]
for $\xi \in \kt \cap \kh$ and $\zeta \in \ka = \kh \cap \kp$ such that $\xi + \zeta \in \kh$ is regular.
Suppose that there is a set of roots $S \subset R(\kg^{\C}, \kh^{\C})$, and for every $\alpha \in S$, there are $X_{\pm \alpha} \in \kg^{\C}_{\pm \alpha}$ such that
\begin{itemize}
\item $X_{\alpha}+ X_{-\alpha} \in \kp$;
\item $\eta_{\alpha} := [X_{\alpha}, X_{-\alpha}] \in i\kt$;
\item $\langle \alpha, \eta_{\alpha}\rangle >0$.
\end{itemize}
Then $\im(\Phi) \cap \kt$ contains the convex hull of
\[
\bigcup_{\alpha \in S} \bigl( \xi + \R_{\geq 0} {\langle \alpha, \xi \rangle } \eta_{\alpha} \bigr).
\]
\end{lemma}
\begin{proof}
As in the proof of Lemma \ref{lem im mu 3}, we find that
for all $t \in \R$,
\begin{multline}\label{eq Phi X alpha}
\Phi(\exp(t(X_{\alpha}+X_{-\alpha}))H) = \cosh(\ad(t(X_{\alpha}+X_{-\alpha})))\xi + \sinh(\ad(t(X_{\alpha}+X_{-\alpha})))\zeta\\
	= \xi + \frac{\langle \alpha, \xi \rangle }{\langle \alpha, \eta_{\alpha} \rangle } 
		\bigl(\cosh(t\sqrt{2\langle \alpha, \eta_{\alpha}\rangle} ) -1\bigr)  \eta_{\alpha}
		- \frac{ \langle \alpha, \zeta \rangle}{\sqrt{2\langle \alpha, \eta_{\alpha} \rangle}} \sinh(t\sqrt{2\langle \alpha, \eta_{\alpha}\rangle} )(X_{\alpha} - X_{-\alpha}).
\end{multline}
The left hand side and the first term on the right hand side lie in $\kg$, hence so does the second term on the right hand side. But
\[
X_{\alpha} - X_{-\alpha} = \frac{1}{\langle \alpha, \eta_{\alpha} \rangle}[\eta_{\alpha}, X_{\alpha}+X_{-\alpha}] \in i\kg.
\]
And $\langle\alpha, \zeta\rangle \in \R$, so the second term on the right hand side of \eqref{eq Phi X alpha} lies in $i\kg\cap\kg  = \{0\}$. 
We conclude that
\[
\Phi(\exp(t(X_{\alpha}+X_{-\alpha}))H) =  \xi + \frac{\langle \alpha, \xi \rangle }{\langle \alpha, \eta_{\alpha} \rangle } 
		\bigl(\cosh(t\sqrt{2\langle \alpha, \eta_{\alpha}\rangle} ) -1\bigr)  \eta_{\alpha} \quad \in \kt.
\]
So
\[
\xi + \R_{\geq 0} {\langle \alpha, \xi \rangle } \eta_{\alpha} \subset \im(\Phi) \cap \kt.
\]
The claim again follows by convexity of $\im(\Phi) \cap \kt$.
%
\end{proof}

\begin{example}
If $G = \SL(2,\R)$, $\xi = \begin{pmatrix} 0 & -1 \\ 1 & 0\end{pmatrix}$, $\zeta = 0$, $\kh = \kt = \R\xi$, $\langle \alpha, \xi \rangle = 2i$, $S = \{\alpha\}$, and
\[
X_{\alpha} =  \frac{1}{2}\begin{pmatrix} 1 & -i \\ -i & -1\end{pmatrix}; \qquad X_{-\alpha} =  \frac{1}{2}\begin{pmatrix} 1 & i \\ i & -1\end{pmatrix},
\]
then $X_{\alpha} + X_{-\alpha} \in \kp$, $\eta_{\alpha} = -i\xi \in i\kt$, and Lemma \ref{lem im mu 1} states that $\im(\Phi)$ contains the half-line 
$[1,\infty)\xi$. (In this case, we actually find that $\im(\Phi)$ equals that half-line.)
\end{example}

\begin{proof}[Proof of Proposition \ref{prop dim im mu}]
For every $\alpha \in R^+_n$, the element $H_{\alpha} = [E_{\alpha}-\bar E_{\alpha}, E_{\alpha}+\bar E_{\alpha}]/2$ is imaginary, it lies in $\kh_c^{\C}$ and in $[\kp^{\C}, \kp^{\C}] \subset \kk^{\C}$. Hence $H_{\alpha} \in i\kt$. 
Therefore, 
applying Lemma \ref{lem im mu 1} with $S = R^+_n$, $X_{\alpha} = E_{\alpha}$ and $X_{-\alpha} = \bar E_{\alpha}$ shows that the claim holds for $\kh = \kh_c$.

Now fix $\alpha \in R^+_n$. Consider the Cayley transform
\[
c_{\alpha} := \Ad\bigl(\exp(\frac{\pi}{4} (\bar E_{\alpha} - E_{\alpha})) \bigr).
\]
(For the properties of Cayley transforms we use, see for example Section VI.7 of \cite{Knapp02}.) Set $\kh_1 := c_{\alpha}(\kh_c)\cap \kg$, and 
\[
\begin{split}
X_{\alpha} &:= ic_{\alpha}(E_{\alpha});\\
X_{-\alpha} &:= -ic_{\alpha}(\bar E_{\alpha}).
\end{split}
\]
These elements lie in root spaces for $\kh_1$. They satisfy
\begin{enumerate}
\item $X_{\alpha} + X_{-\alpha} = i(E_{\alpha} - \bar E_{\alpha}) \in \kg$;
\item $X_{\alpha} - X_{-\alpha} = ic_{\alpha}(E_{\alpha} + \bar E_{\alpha}) = -iH_{\alpha} \in \kt$;
\item $[X_{\alpha}, X_{-\alpha}] = c_{\alpha}(H_{\alpha}) = E_{\alpha} + \bar E_{\alpha}\in \kh_{1}\cap \kp$.
\end{enumerate}
Hence Lemma \ref{lem im mu 3} implies that, with $\Phi$ as in the proposition for $\kh = \kh_1$,
\[
\xi + i\R H_{\alpha} \in \im(\Phi)\cap \kt.
\]
As in the first paragraph of this proof, by applying 
 Lemma \ref{lem im mu 1} with $S = R^+_n \setminus \{\alpha\}$, we find that
\[
\bigcup_{\alpha \in R^+_n \setminus \{\alpha\}} \xi + I_{\alpha}iH_{\alpha} \subset \im(\Phi)\cap \kt, 
\]
with $I_{\alpha}$ equal to $[0,\infty)$ or $(-\infty, 0]$. If $\xi + \zeta$ is regular, then $\Phi$ is a moment map in the synplectic sense, so its image intersected with $\kt$ is convex. Hence the claim follows for $\kh = \kh_1$.

Continuing in this way, removing noncompact, imaginary roots until there are none left, one proves the claim for all $\theta$-stable Cartan subalgebras.
\end{proof}

\subsection{Examples: $\SU(p,1)$, $\SO_0(p, 1)$ and $\SO_0(2,2)$} \label{sec ex mult free}

\begin{lemma}\label{lem im mu SUpq}
Let $G = \SU(p,q)$. Let $H<G$ be a $\theta$-stable Cartan subgroup, and let $\mu \in \kh$ be regular. The image of the map 
\[
\Phi\colon G/H \cong \Ad^*(G)\mu \hookrightarrow \kg \to \kk.
\]
 has nonempty interior.
\end{lemma}
\begin{proof}
Let $H_c<G$ be the compact Cartan of diagonal elements. Then a choice  of positive imaginary noncompact roots of $(\kg^{\C}, \kh^{\C}_c)$ is 
\[
R^+_n = \{\alpha_{jk}; 1\leq j \leq p, p+1\leq k \leq p+q\}, 
\]
where $\alpha_{jk}$ maps the diagonal matrix with entries $(t_1, \ldots, t_{p+q})$ to $t_j - t_k$. A root vector in $\kg^{\C}_{\alpha_{jk}}$ is the matrix $E_{jk}$ win a $1$ in position $(j,k)$ and zeroes in the other positions. The complex conjugation of $E_{jk}$ with respect to the real form $\su(p,q)$ is $E_{kj}$.  And
\[
[E_{jk}, E_{kj}] = h_{jk},
\]
where $h_{jk}$ is the diagonal matrix with entry
with $1$ in the $j$th position and $-1$ in the $k$th position, and zeroes everywhere else. Together, these span $i \kh_c$. So Proposition \ref{prop dim im mu} implies that $\im(\Phi)\cap \kt$ has nonempty interior in $\kt$, so that $\im(\Phi)$ has nonempty interior in $\kk$.
\end{proof}

\begin{lemma}\label{lem im mu SO0pq}
Let $G = \SO_0(p,q)$, with $p$ and $q$ even. Let $H<G$ be a $\theta$-stable Cartan subgroup, and let $\mu \in \kh$ be regular. The image of the map 
\[
\Phi\colon G/H \cong \Ad^*(G)\mu \hookrightarrow \kg \to \kk.
\]
 has nonempty interior.
\end{lemma}
\begin{proof}
Write $p=2r$, $q=2s$ and $l=r+s$. Consider the compact Cartan subgroup
 $H_c = \SO(2)^{l}<G$. For $j=1, \ldots, l$, let $h_j \in H_c$ be the matrix with a block $X = \begin{pmatrix}0 & -1 \\ 1 & 0 \end{pmatrix}$ as the $j$th $2\times 2$ block on the diagonal, and zeroes everywhere else. For $j,k=1, \ldots, l$, with $j<k$, define positive roots $\alpha_{jk}^{\pm}$ by 
 \[
 \begin{split}
 \langle \alpha_{jk}^{\pm}, h_j \rangle &= i\\
  \langle \alpha_{jk}^{\pm}, h_k \rangle &= \pm i,
 \end{split}
 \]
and $\langle \alpha_{jk}^{\pm}, h_m\rangle = 0$ for all other $m$. Then 
\[
R^+_n = \{\alpha_{jk}^{\pm}; 1 \leq j \leq r, r+1 \leq k \leq l\}
\]
 is a choice of positive, noncompact, imaginary roots. A root vector for $\alpha^{\pm}_{jk}$ is the matrix $E_{\alpha^{\pm}_{jk}}$ with a $2\times 2$ block
\[
Y_{\pm}=
\frac{1}{2}
\begin{pmatrix}
1 & \mp i \\ -i & \mp 1
\end{pmatrix}
\]
as the $2\times 2$ block in position $(j,k)$ and a block $-Y_{\pm}^T$ in position $(k,j)$, if we divide $n\times n$ matrices into $l\times l$ blocks of size $2\times 2$. 
 And
\[
[E_{\alpha^{\pm}_{jk}}, \bar E_{\alpha^{\pm}_{jk}}] = i (h_j \pm h_k).
\]
The set
\[
 \{h_j \pm h_k; 1 \leq j \leq r, r+1 \leq k \leq l\}
 \] 
 spans $\kh_c$. So Proposition \ref{prop dim im mu} implies that $\im(\Phi)\cap \kt$ has nonempty interior in $\kt$, so that $\im(\Phi)$ has nonempty interior in $\kk$.
\end{proof}

\begin{lemma}\label{lem im mu SO0p1}
Let $G = \SO_0(p,1)$. Let $H<G$ be a $\theta$-stable Cartan subgroup, and let $\mu \in \kh$ be regular. The image of the map 
\[
\Phi\colon G/H \cong \Ad^*(G)\mu \hookrightarrow \kg \to \kk.
\]
 has nonempty interior.
\end{lemma}
\begin{proof}
Write $p=2l$ or $p=2l+1$ depending on the parity of $p$. Set 
\[
A := \Bigl\{\begin{pmatrix}
\cosh(t) & \sinh(t) \\ \sinh(t) & \cosh(t)
\end{pmatrix}; t \in \R \Bigr\}
\]
Consider the maximal torus $T = \SO(2)^l$ of $K = \SO(p)$. A maximally compact Cartan subgroup of $G$ is  $H_c = T$ if $p$ is even, and $H_c = T\times A$ if $p$ is odd.

For $j=1,\ldots, l$, let $h_j$ be the matrix whose $j$th $2\times 2$ block on the diagonal is $ \begin{pmatrix}0 & -1 \\ 1 & 0 \end{pmatrix}$, and with all other entries zero.
Consider the root $\alpha_j$ of $(\kg^{\C}, \kh_c^{\C})$ given by $\langle \alpha_j, h_j \rangle = i$, $\langle \alpha_j, h_k \rangle = 0$ if $k\not=j$, and, if $p$ is odd, $\alpha_j|_{\ka}=0$. A root vector for $\alpha_j$ is the matrix
\[
E_{\alpha_j} = \begin{pmatrix}
0 & \cdots & & & \cdots & 0 \\
\vdots & & & & & \vdots \\
 & & & & & -i \\
 & & & & & 1 \\
\vdots  & & & & & \vdots \\
0 & \cdots & i & -1 & \cdots & 0
\end{pmatrix},
\]
where the two  nonzero entries in the last column are in rows $2j-1$ and $2j$, and the two nonzero entries in the bottom row are in columns $2j-1$ and $2j$. So $\alpha_j$ is an imaginary, noncompact root. The matrices $[E_{\alpha_j}, \bar E_{\alpha_j}] = -2i h_j$, where $j = 1,\ldots, l$, span $i\kt$.  Proposition \ref{prop dim im mu} implies that $\im(\Phi)\cap \kt$ has nonempty interior in $\kt$, so that $\im(\Phi)$ has nonempty interior in $\kk$.
\end{proof}

Combining Corollary \ref{cor mult free} and Lemmas \ref{lem im mu SUpq}--\ref{lem im mu SO0p1} with the list of groups  in Subsection \ref{sec dim count}, we obtain the following consequence of Theorem \ref{thm mult form}.
\begin{corollary}\label{cor mult free ex}
If $G = \SU(p,1)$, $G = \SO_0(p,1)$ or $G = \SO_0(2,2)$, then  any admissible representation of $G$ has multiplicity-free restriction to a maximal compact subgroup.
\end{corollary}
Koornwinder \cite{Koornwinder82} proved the cases $G = \SU(p,1)$ and $G = \SO_0(p,1)$. We give a geometric explanation of this fact here, include the case $G = \SO_0(2,2)$, and also a geometric criterion for when multiplicites equal one (see Corollary \ref{cor one 2}). Using Proposition \ref{prop dim im mu}, one can investigate the groups listed at the start of this section in a similar way. 

Note that $\SU(p,q)$ is of Hermitian type (meaning that $G/K$ is a Hermitian symmetric space), but $\SO_0(p,1)$ and $\SO_0(2,2)$ are not. Therefore, Corollary \ref{cor mult free ex} illustrates the fact that our method applies beyond the Hermitian case considered for example in 
 \cite{Kobayashi97}. Furthermore, $\SO_0(p,1)$ has no discrete series for $p$ odd, so that we find that the method yields nontrivial results for such groups as well.

\bibliographystyle{plain}
\bibliography{mybib}

\end{document}